\documentclass[11pt]{article}

\usepackage{amsmath,amssymb,amsfonts,enumitem,amsthm,pstricks}

\numberwithin{equation}{section}

\newtheorem{thm}{Theorem}[section]
\newtheorem{prp}[thm]{Proposition} 
\newtheorem{lmm}[thm]{Lemma} 
\newtheorem{cnj}[thm]{Conjecture} 
\newtheorem{ques}[thm]{Question} 
 
\newtheorem{crl}[thm]{Corollary} 
\theoremstyle{remark}
\newtheorem{rmk}[thm]{Remark} 

\def\sf#1{\textsf{#1}}

\def\BE#1{\begin{equation}\label{#1}}
\def\EE{\end{equation}}
\def\eref#1{(\ref{#1})}
\def\ti#1{\tilde{#1}}
\def\ov#1{\overline{#1}}
\def\lr#1{\langle{#1}\rangle}
\def\blr#1{\big\langle{#1}\big\rangle}
\def\wt#1{\widetilde{#1}}
\def\sm#1{\begin{small}#1\end{small}}

\def\lra{\longrightarrow}
\def\i{\infty}

\def\al{\alpha}
\def\be{\beta}
\def\de{\delta}
\def\ep{\epsilon}
\def\la{\lambda}
\def\om{\omega}

\def\th{\theta}

\def\Ga{\Gamma}

\def\C{\mathbb C}
\def\cC{\mathcal C}
\def\bE{\mathbb E}
\def\cM{\mathcal M}
\def\P{\mathbb P}
\def\Q{\mathbb Q}
\def\R{\mathbb R}
\def\Z{\mathbb Z}

\def\fM{\mathfrak M}

\def\d{\textnormal d}
\def\fI{\mathfrak i}
\def\h{\textnormal{h}}
\def\v{\textnormal{v}}

\def\ev{\textnormal{ev}}
\def\Im{\textnormal{Im}}
\def\Re{\textnormal{Re}\,}

\begin{document}

\title{Some Conjectures on\\ 
the Asymptotic Behavior of Gromov-Witten Invariants}
\author{Aleksey Zinger\thanks{Partially supported by NSF grants DMS 0846978 and 1500875}}
\date{\today}
\maketitle

\begin{abstract}
\noindent
The purpose of this note is to share some observations and speculations concerning
the asymptotic behavior of Gromov-Witten invariants.
They may be indicative of some deep phenomena
in symplectic topology that in full generality are outside of the reach of current techniques. 
On the other hand, many interesting cases can perhaps be treated via combinatorial techniques.
\end{abstract}

\tableofcontents

\section{Introduction}
\label{intro_sec}

\noindent
Gromov-Witten invariants are certain counts of curves in smooth projective varieties
and of pseudoholomorphic maps into symplectic manifolds inspired by~\cite{Gr}.
Their applications in symplectic topology have included 
Gromov's Non-Squeezing Theorem \cite[Theorem~9.3.1]{MS},
distinguishing diffeomorphic symplectic manifolds \cite{Ruan94},
and the uniruledness of symplectic manifolds with Hamiltonian group actions~\cite{Mc09}.
On the other hand, the vast literature on Gromov-Witten invariants in algebraic geometry
has generally concerned striking algebraic properties exhibited by collections
of these invariants associated with individual manifolds.
These properties have included the associativity of the quantum product 
on the cohomology, 
mirror symmetry in genus~0 \cite{Gi,LLY} and in genus~1 \cite{bcov1ci,bcov1},
and the modularity of some generating functions for Gromov-Witten invariants
\cite{BL,Tzeng}.\\

\noindent
The present note concerns three types of asymptotic behavior of Gromov-Witten invariants
as the degree of the curves and the energy of the maps being counted increase:
upper bounds, asymptotics involving upper and lower bounds, and vanishing statements.
While observations in this spirit were first made in the early days of 
Gromov-Witten theory \cite{FI,T95},
there has been fairly little progress on these kinds of problems since~then.
Developments in the theory over the past 25~years have reduced such problems
in many important cases to essentially combinatorial questions.
Recent works, some of which we review in the present note, have tackled these questions
in some cases and appear very promising for treating them in many other cases.
While combinatorial techniques should confirm many important special cases 
of the conjectures stated in Section~\ref{main_sec},
they are unlikely to explain the geometry behind the conjectured phenomena on their~own  though.
These phenomena may have connections with Strominger-Yau-Zaslow's deep proposal \cite{SYZ}
and Gross-Siebert's related program~\cite{GS03} suggesting that 
the Gromov-Witten invariants of (at least many) symplectic manifolds are
``made up" of the Gromov-Witten invariants of some toric pieces.\\

\noindent
A conjecture for rather precise asymptotics of the arbitrary-genus Gromov-Witten invariants
of the complex projective plane~$\P^2$ was formulated in~\cite{FI}.
In Section~\ref{asympt_subs}, we recall this conjecture,
summarize the recent work \cite{TW,AsympGWnotes} completing the proof
of its genus~0 case initiated in~\cite{FI} and moving on to the genus~1 case, 
and discuss extensions to other manifolds.
In Section~\ref{EBounds_subs}, we state a conjecture bounding
Gromov-Witten invariants by the energy of the underlying maps,
summarize an approach for establishing it in some cases,
and describe its connection with the proposal of~\cite{MaP}
for showing that natural generating functions for Gromov-Witten invariants
of many symplectic manifolds have nonzero radii of convergence.
In Section~\ref{vanish_subs}, we state a vanishing conjecture for 
the (descendant) Gromov-Witten invariants of monotone symplectic manifolds
(and smooth Fano varieties) and again summarize an approach for establishing it in some cases.
Section~\ref{details_sec} contains a detailed proof of the genus~0 and~1 
cases of the conjecture of~\cite{FI} stated in Section~\ref{asympt_subs};
it combines parts of~\cite{FI}, \cite{AsympGWnotes}, and~\cite{TW}.
We hope that this proof can be adapted to many other manifolds.
Section~\ref{misc_sec} contains related miscellaneous observations.\\

\noindent
The author would like to thank R.~Pandharipande for bringing up~\cite{FI} and 
for sharing his idea to establish the bound~\eref{CY3bnd_e},
G.~Tian for raising questions about asymptotics in Gromov-Witten theory,
P.~Sarnak and D.~Grigoriev for enlightening discussions,
A.~Gathmann for the {\it growi} program computing Gromov-Witten invariants,
and the IAS School of Mathematics for hospitality while 
\cite{g0ci} and \cite{AsympGWnotes} were completed.

\section{Conjectures and theorems}
\label{main_sec}

\noindent
The three parts of this section contain conjectures concerning asymptotic behavior
of Gromov-Witten invariants of three different flavors.
The first conjecture is particularly simple to state,
as it concerns only curve counts in the complex projective plane~$\P^2$.
The last two conjectures require a nominal amount of notation.
If true, they should be consequences of fundamental properties of Gromov-Witten invariants
that are yet to be discovered.

\subsection{Asymptotic expansions}
\label{asympt_subs}

\noindent
In the case of the complex projective plane~$\P^2$,
the classical enumerative counts of complex curves and 
the modern Gromov-Witten invariants agree.
For $g\!\in\!\Z^{\ge0}$ and $d\!\in\!\Z^+$, we denote by~$N_{g,d}$ 
the number of degree $d$ genus~$g$ curves through $3d\!-\!1\!+\!g$ general points in~$\P^2$.

\begin{cnj}[{\cite[Footnote 2]{FI}}]\label{P2_cnj}
There exist $b\!\in\!\R^+$ independent of $g$ and $a_g\!\in\!\R^+$ for each $g\!\in\!\Z^{\ge0}$
such~that 
$$\frac{N_{g,d}}{(3d\!-\!1\!+\!g)!}=a_gb^dd^{-1-\frac52(1-g)}\big(1+o(1)\big)
\qquad\hbox{as}\quad d\lra\i$$
for each $g\!\in\!\Z^{\ge0}$.
\end{cnj}

\noindent
The $g\!=\!0$ case of Conjecture~\ref{P2_cnj} is \cite[Proposition 3]{FI}.
The starting point for the reasoning behind \cite[Proposition 3]{FI}
is Kontsevich's recursion for~$N_{0,d}$, restated as~\eref{P2g0dfn_e} in the present note.
It is claimed in~\cite{FI} that the $g\!=\!0$ case of Conjecture~\ref{P2_cnj} 
is a direct consequence of 
the statement of Proposition~\ref{FIexp_prp} in the present note.
As noted in~\cite{AsympGWnotes},
\begin{enumerate}[label=(G\arabic*),leftmargin=*]

\item\label{F0exp_it} the existence of the expansion~\eref{F0exp_e} used
to establish the conclusion of Proposition~\ref{FIexp_prp} was not justified in~\cite{FI};

\item\label{F0impl_it} Proposition~\ref{FIexp_prp} does not directly imply the $g\!=\!0$ case of
Conjecture~\ref{P2_cnj} or even the first statement of
Corollary~\ref{TWexp_crl} below;

\end{enumerate}
see Section~\ref{FI_subs} for details.
Following P.~Sarnak's suggestion to use the Frobenius method to bypass~\ref{F0exp_it},
\cite{AsympGWnotes} established Proposition~\ref{FIexp_prp} and suggested 
the use of the Eguchi-Hori-Xiong recursion \cite[(8)]{Pand99} to move from
the genus~0 case of Conjecture~\ref{P2_cnj} to the genus~1 case.
Lemmas~\ref{TW_lmm0} and~\ref{TW_lmm} in the present note
were first obtained (in slightly different formulations)  
in~\cite{TW} with the knowledge of the contents of~\cite{AsympGWnotes}
and completed the proof of the genus~0 and~1 cases of Conjecture~\ref{P2_cnj},
with even more refined statements for the asymptotics of~$N_{0,d}$ and~$N_{1,d}$.

\begin{thm}[{\cite[Theorems~1.1,1.2]{TW}}]\label{TW_thm}
There exist $b\!\in\!\R^+$, 
$a_{0;k},a_{1;k}\!\in\!\R$ for each $k\!\in\!\Z^{\ge0}$, and 
$C_N\!\in\!\R$ for each $N\!\in\!\Z^+$ such that
\begin{equation*}\begin{split}
\bigg|\frac{N_{0,d}}{(3d\!-\!1)!}
-b^d\!\sum_{k=3}^{N-1}\!a_{0;k}d^{-k-\frac12}\bigg|&<C_Nb^dd^{-N-\frac12}\,,\\
\bigg|\frac{N_{1,d}}{(3d)!}-b^d
\bigg(\frac{1}{48d}+\sum_{k=0}^{N-1}\!a_{1;k}d^{-k-\frac12}\!\!\bigg)\bigg|&<C_Nb^dd^{-N-\frac12}
\end{split}\end{equation*}
for all $d,N\!\in\!\Z^+$.
\end{thm}

\begin{crl}[{\cite[Corollaries~1.1,1.2]{TW}}]\label{TWexp_crl}
There exists $b\!\in\!\R^+$ such~that 
\BE{TWexpcrl_e}\lim_{d\lra\i}\sqrt[d]{\frac{N_{0,d}}{(3d\!-\!1)!}}=b=
\lim_{d\lra\i}\sqrt[d]{\frac{N_{1,d}}{(3d)!}}\,;\EE
in particular, the above limits exist and are nonzero.
\end{crl}

\noindent 
The Eguchi-Hori-Xiong recursion, restated as~\eref{P2g1dfn_e} in the present note, 
determines the numbers~$N_{1,d}$ from the numbers~$N_{0,d}$.
As noted in~\cite{AsympGWnotes},
the second equality in~\eref{TWexpcrl_e} is a direct consequence of the first
because of this recursion; see the end of Section~\ref{FI_subs}.
This observation motivated the transition from the $g\!=\!0$ to the $g\!=\!1$ 
case in~\cite{TW}; see Corollary~\ref{F1_crl} and Lemma~\ref{TW_lmm0} in the present note.\\

\noindent
Corollary~\ref{TWexp_crl} is an immediate consequence of Theorem~\ref{TW_thm}.
We give a detailed proof of Theorem~\ref{TW_thm} in Section~\ref{details_sec}
to encourage others to consider similar asymptotics and energy bounds problems.
Section~\ref{FI0_subs} presents the approach of~\cite{FI} for bounding recursively
defined sequences by essentially geometric sequences.
Section~\ref{FI_subs} describes the remainder of the reasoning of~\cite{FI}
behind the claim to establish the genus~0 case of Conjecture~\ref{P2_cnj}
and discusses the gaps~\ref{F0exp_it} and~\ref{F0impl_it}.
Section~\ref{TW_subs} is a more systematic, though mathematically analogous, version 
of the treatment of~\ref{F0impl_it} provided by~\cite{TW}.
Section~\ref{expand_subs} contains a proof of Proposition~\ref{FIexp_prp}, 
which addresses~\ref{F0exp_it};
it is essentially the same argument as in~\cite{AsympGWnotes}, but is better organized.\\

\noindent
By the proof of Theorem~\ref{TW_thm}, the asymptotics for~$N_{0,d}$ and $N_{1,d}$ 
are completely determined by~$b$ in Conjecture~\ref{P2_cnj} and 
by the coefficients~$a_0$ and~$a_2$ in~\eref{FIexp_e}.
The remaining coefficients in~\eref{FIexp_e} are determined by~\eref{arec_e1} and~\eref{arec_e2}
and in turn determine the asymptotic coefficients~$a_{0;k}$ and~$a_{1;k}$ in 
Theorem~\ref{TW_thm} via Lemma~\ref{TW_lmm} and~\eref{F1cond_e}.
It would be interesting to determine these coefficients explicitly
and to understand their geometric meaning.
The leading coefficient in the genus~1 asymptotic expansion, 
i.e.~$a_1$ in Conjecture~\ref{P2_cnj}, is half the first Chern class of
the Hodge line bundle~$\bE_1$ on the Deligne-Mumford moduli space~$\ov\cM_{1,1}$ of elliptic curves.
This suggests a potential connection between the asymptotics conjectured in~\cite{FI}
in the early days of Gromov-Witten theory and the behavior of Gromov-Witten invariants 
under degenerations of the~target.\\

\noindent
The counts of genus~0 curves in many standard K\"ahler manifolds agree with
the corresponding Gromov-Witten invariants and can be computed recursively.
Asymptotics for the growth of such counts for blowups of~$\P^2$ are obtained
in \cite{IKS04,IKS05}, but these asymptotics are fairly coarse.
The~$\P^2$ case of these asymptotics is essentially equivalent
to the existence of
the lower and upper bounds in~\eref{P2asymp_e0}, which were originally obtained in~\cite{FI}.
On the other hand, the recursions determining the counts of genus~0 curves in blowups of~$\P^2$
have the same general structure as Kontsevich's recursion \cite[(10.4)]{RT} for~$\P^2$.
As shown in~\cite{Pand99}, the Eguchi-Hori-Xiong recursion for counts of genus~1 curves in~$\P^2$
is essentially a lift of Getzler's relation \cite{Getz} from $\ov\cM_{1,4}$ 
to the moduli space of genus~1 stable maps to~$\P^2$.
It thus has analogues for other K\"ahler and more generally symplectic manifolds. 
This all suggests that the approach described in Section~\ref{details_sec} may extend from~$\P^2$ 
to obtain refined asymptotics for counts of genus~0 and~1 curves in its blowups.\\

\noindent
A completely different approach is needed to deal with the $g\!\ge\!2$ cases of 
Conjecture~\ref{P2_cnj} and its potential analogues for other K\"ahler surfaces.
Some kind of geometric degeneration approach remains elusive at this point.
On the other hand, the G\"ottsche-Yau-Zaslow formula \cite{Tzeng} essentially
enumerates arbitrary-genus curves on smooth algebraic surfaces and appears promising
as the starting point for a combinatorial approach to the $g\!\ge\!2$ cases of
Conjecture~\ref{P2_cnj} and its potential generalizations.\\

\noindent
For complex manifolds of dimension~3 and higher, different types of incidence conditions 
(not just points) should be considered.
All counts of genus~0 curves in~$\P^n$ are computable from the recursion of 
\cite[Theorem~10.4]{RT}.
For $\P^3$, {\it Mathematica} suggests Conjecture~\ref{P3g0_cnj} below;
it is based on the numbers up to $d\!=\!200$ 
(the computation of these numbers already takes a long time).
As the convergence appears to be very slow
(for $N_{0,d}$, it is still going noticeably even for $d\!=\!1000$),
it is feasible that the limit below is even independent of the slope~$\al/\be$ chosen, 
but the numbers so far do not suggest this.

\begin{cnj}\label{P3g0_cnj}
Let $N_{0,d}(p)$ be the number of degree~$d$ rational curves
through $2d\!-\!p$ points and $2p$ lines in general position in~$\P^3$. 
For all $\al,\be\!\in\!\Z^+$ fixed, the~limit
$$\lim_{d\lra\i}\sqrt[d]{\frac{N_{0,\al d}(\be d)}{((2\al\!+\!\be)d)!}}$$
exists and is nonzero.
\end{cnj}

\noindent
An upper bound on the sequences in Conjecture~\ref{P3g0_cnj} can be obtained
from a two-variable version of the approach used in the proof of 
\cite[Proposition 3]{FI}; see Section~\ref{P3g0_subs}.
It also follows immediately from \cite[Theorem~1]{g0ci}.
A lower bound appears more elusive, since 
the recursion of \cite[Theorem~10.4]{RT} for $\P^3$ involves 
negative coefficients.\\

\noindent
A natural extension of Conjecture~\ref{P3g0_cnj} to arbitrary genera~$g$ and 
arbitrary compact symplectic manifolds~$(X,\om)$ would be as follows.
Fix $H_1,\ldots,H_k\!\in\!H^*(X)$.
Let $\be_r\!\in\!H_2(X)$ be a sequence such that $\lr{\om,\be_r}\!\lra\!\i$
and the lines $\R\be_r$ converge in the projectivization of~$H_2(X;\R)$.
Suppose $b_{1;r},\ldots,b_{k;r}\!\in\!\Z^+$ are sequences such that 
the lines $[b_{1;r},\ldots,b_{k;r}]$ converge in $\R\P^{k-1}$ and
$$\sum_{i=1}^kb_{i;r}\big(\!\deg H_i\big)=2\blr{c_1(TX),\be_r}
+\big(\dim_{\R}X\!-\!6\big)(1\!-\!g)+2k
\qquad\forall\,r\!\in\!\Z^+\,.$$
One might then ask whether the sequence of Gromov-Witten invariants
$$\bigg(\blr{\underset{b_{1;r}}{\underbrace{H_1,\ldots,H_1}},\ldots,
\underset{b_{k;r}}{\underbrace{H_k,\ldots,H_k}}}_{g,\be_r}^X\bigg)^{1/\lr{\om,\be_r}}$$
converges; see \eref{GWdfn_e} for the notation.

\subsection{Energy bounds}
\label{EBounds_subs}

\noindent
Enumerative counts of curves in many smooth projective varieties
are not well-defined.
It is instead natural to consider the asymptotic behavior of Gromov-Witten invariants.
For $g,N\!\in\!\Z^{\ge0}$, an almost K\"ahler manifold $(X,\om,J)$, and $\be\!\in\!H_2(X)$,
we denote~by $\ov\fM_{g,N}(X,\be)$ the moduli space of stable 
degree~$\be$ genus~$g$  $N$-marked  $J$-holomorphic maps to~$X$. 
For each $s\!=\!1,\ldots,N$, let
$$\ev_s\!: \ov\fM_{g,N}(X,\be)\lra X \quad\hbox{and}\quad
\psi_s\equiv c_1(L_s^*)\in H^2\big(\ov\fM_{g,N}(X,\be);\Q\big)$$
be the evaluation map 
and the first Chern class of  the universal cotangent line bundle at the $s$-th marked point,
respectively.
For $b_1,\!\ldots,b_N\!\in\!\Z^{\ge0}$ and $H_1,\ldots,H_N\!\in\!H^*(X;\Q)$, let
\BE{GWdfn_e}\blr{\tau_{b_1}H_1,\ldots,\tau_{b_N}H_N}_{g,\be}^X
= \int_{[\ov\fM_{g,N}(X,\be)]^{vir}}    
\prod_{s=1}^{s=N}\!\!\big(\psi_s^{b_s}\ev_s^*H_s\big);\EE
this rational number is a \textsf{descendant Gromov-Witten invariant}.

\begin{cnj}[{\cite[Conjecture~1]{g0ci}}]\label{GWbound_cnj}
Suppose $(X,\om)$ is a compact symplectic manifold and $g\!\in\!\Z$.
For all $H_1,\ldots,H_k\!\in\!H^*(X)$, there exists $C_{X,g}\!\in\!\R^+$ such~that 
\BE{GWbound_e}\bigg|\frac{\blr{b_1!\,\tau_{b_1}H_{c_1},\ldots,b_N!\,\tau_{b_N}H_{c_N}}_{g,\be}^X}{ N!}
\bigg|\le C_{X,g}^{\lr{\om,\be}+N}\EE
for all $\be\!\in\!H_2(X)$, $N,b_s\!\in\!\Z^{\ge0}$, and $c_s\!\in\!\{1,\ldots,k\}$.
\end{cnj}

\noindent
The exponent $\lr{\om,\be}$  in Conjecture~\ref{GWbound_cnj}
is the energy of the $J$-holomorphic maps of class~$\be$, while
$\lr{\om,\be}\!+\!N$ is essentially the energy of the induced ``graph map".
The $\be\!=\!0$ case of this conjecture is obtained by induction
from the string and dilaton equations \cite[p527]{MirSym}
for Hodge classes on the Deligne-Mumford moduli spaces $\ov\cM_{g,k}$ 
of stable $k$~marked genus~$g$ curves.
In fact, the limsup of the $N$-th root of the left-hand side in~\eref{GWbound_e}
is at most~1.
Theorem~1 in~\cite{posgen} combines these relations with the Virtual Equivariant Localization
Theorem of~\cite{GP} to establish Conjecture~\ref{GWbound_cnj} for $X\!=\!\P^n$.
Theorem~1 in~\cite{g0ci} establishes the $g\!=\!0$ case of 
this conjecture for complete intersections $X\!\subset\!\P^n$
with each $H_s$ being a power of the hyperplane class $H\!\in\!H^2(\P^n)$.
For the reasons explained below, Conjecture~\ref{GWbound_cnj} also holds  
for all Calabi-Yau complete intersections $X\!\subset\!\P^n$ of dimension at least~4
with each $H_s$ again being a power of~$H$.\\

\noindent
For a Calabi-Yau threefold~$X$, Conjecture~\ref{GWbound_cnj} is equivalent to 
the existence of $C_{X,g}\!\in\!\R^+$ such~that 
\BE{CY3bnd_e}\big|\blr{}_{g,\be}^X\big|\le C_{X,g}^{\lr{\om,\be}}\qquad\forall~
\be\!\in\!H_2(X)\,.\EE
The existence of such a $C_{X,g}$ in turn corresponds to 
the string theory presumption that the  partition function determined by 
the genus~$g$ Gromov-Witten invariants of~$X$ has positive radius of convergence.
The bounds~\eref{CY3bnd_e} are  implied by mirror symmetry predictions.
These predictions have been confirmed mathematically for $g\!=\!0$
for many Calabi-Yau threefolds in \cite{Gi,LLY,LLY2,LLY3} and 
for $g\!=\!1$ for  Calabi-Yau complete intersections $X\!\subset\!\P^n$ in \cite{bcov1ci,bcov1},
but are yet to be confirmed  for $g\!\ge\!2$ for 
any compact Calabi-Yau threefold.
The mirror symmetry predictions can also be used to obtain asymptotics for 
the invariants on the left-hand side of~\eref{CY3bnd_e} in the style of 
Conjecture~\ref{P2_cnj}, as is done in~\cite{CSV,KZ}.\\

\noindent
For a Calabi-Yau manifold~$X$ of (complex) dimension at least~4,
the Gromov-Witten invariants of genus~2 and higher vanish.
Conjecture~\ref{GWbound_cnj} then reduces to its cases for the genus~0 Gromov-Witten invariants
with arbitrary insertions
and for the genus~1 Gromov-Witten invariants with no insertions, i.e.~as in~\eref{CY3bnd_e}.
For complete intersections $X\!\subset\!\P^n$,
such genus~0 bounds with each $H_i$ being a power of~$H$ are provided by 
Theorem~1 in~\cite{g0ci}.
The required genus~1 bounds for complete intersections $X\!\subset\!\P^n$ 
are implied by the genus~1 mirror formulas established in~\cite{bcov1ci,bcov1}.\\

\noindent
Theorem~1 in~\cite{g0ci} and Theorem~1 in~\cite{posgen} referenced above are obtained 
from the mirror symmetry formulas for the equivariant multi-pointed  genus~0 
Gromov-Witten invariants of complete intersections $X\!\subset\!\P^n$ and 
for the equivariant multi-pointed  genus~$g$ Gromov-Witten invariants of~$\P^n$,
respectively, established in the two papers.
The starting inputs for both mirror formulas are the mirror formulas for
Givental's $J$-function for the equivariant one-pointed genus~0 Gromov-Witten invariants
of complete intersections $X\!\subset\!\P^n$ obtained in \cite{Gi,LLY}
and for its two-pointed analogue obtained in~\cite{bcov0} from Givental's $J$-function.
Givental's $J$-functions for the equivariant one-pointed genus~0 Gromov-Witten invariants
of many other spaces have been computed in \cite{Ciocan95,Kim99,LLY,LLY2,LLY3} and in other works.
The reasoning in \cite{bcov0,g0ci,posgen} can be used to convert these $J$-functions
into 
mirror formulas for the equivariant multi-pointed  genus~0 Gromov-Witten invariants 
of complete intersections in many spaces with groups actions and
the equivariant multi-pointed  genus~$g$ Gromov-Witten invariants of
these spaces themselves.
It should then be possible to establish the corresponding cases of 
Conjecture~\ref{GWbound_cnj} as in \cite{g0ci,posgen}.\\

\noindent
A special case of the bound~\eref{GWbound_e} is \cite[(7.3)]{T95}.
This case concerns only primary Gromov-Witten invariants, 
i.e.~$b_s\!=\!0$ for all~$s$, and only in symplectic manifolds $(X,\om)$
with $H_2(X;\Z)\!=\!\Z$.
It is stated, without a proof, that such a bound follows from the WDVV recursion
(analogue of Kontsevich's recursion) of \cite[Theorem~7.1]{T95}.
This does appear to be the case in general, though we do not see a simple argument.
We provide a proof of \cite[(7.3)]{T95} for $X\!=\!\P^3$ from the WDVV recursion 
by adapting the approach of~\cite{FI} described in Section~\ref{FI0_subs}.\\

\noindent
The paper~\cite{FI} was in fact first brought to the author's attention by R.~Pandharipande in~2008
while describing a scheme~\cite{MaP} to reduce the bounds~\eref{CY3bnd_e} for 
the renown quintic threefold to the bound of Conjecture~\ref{GWbound_cnj} 
for the genus~0 Gromov-Witten invariants of~$\P^3$ via \cite[Theorem~1]{Gi01} and
the degeneration approach of~\cite{MaP0}.
It was an expectation of R.~Pandharipande that the bounds for~$\P^3$ could
be established by adapting the approach of~\cite{FI}.
However, recursive formulas for the descendant invariants of Conjecture~\ref{GWbound_cnj}
involve negative coefficients; this makes the approach of~\cite{FI} unsuitable
for these invariants.
The proof of \cite[Theorem~1]{g0ci}, which in particular
establishes these bounds for the genus~0 Gromov-Witten invariants of all projective spaces,
instead uses the classical equivariant localization theorem of~\cite{AB}.
The proof of \cite[Theorem~1]{posgen}, which extends such bounds to arbitrary genus, 
removes the need to use \cite[Theorem~1]{Gi01} in the scheme of~\cite{MaP} 
to establish the bounds~\eref{CY3bnd_e}.
It also enables the application of this scheme for establishing Conjecture~\ref{GWbound_cnj}
for many other smooth projective varieties
(at least  with each $H_s$  being a power the hyperplane class~$H$).

\begin{ques}\label{bdn_ques}
Are the bounds arising from Conjecture~\ref{GWbound_cnj} sharp?
Are they reflexive of the asymptotic behavior of some natural sequences of
Gromov-Witten invariants, as in Conjectures~\ref{P2_cnj} and~\ref{P3g0_cnj}?
\end{ques}

\subsection{Vanishing statements}
\label{vanish_subs}

\noindent
The observations and speculations on the vanishing of certain descendant Gromov-Witten
invariants in this section concern \sf{monotone} symplectic manifolds.
We recall that a symplectic manifold~$(X,\om)$ is called 
\sf{monotone with minimal Chern number $\nu\!\in\!\R^+$} if 
\BE{c1omcond_e}c_1(X)=\la[\om]\in H^2(X;\R)\EE
for some $\la\!\in\!\R^+$ and $\nu$ is the minimal value of $c_1(X)$
on the homology classes representable by non-constant $J$-holomorphic maps $S^2\!\lra\!X$
for every $\om$-compatible almost complex structure on~$X$.
Perhaps the monotone condition in Conjecture~\ref{GWeq0_cnj} below can be weakened
to $(X,\om)$ being positive (Fano) with minimal Chern number~$\nu$,
i.e.~dropping the requirement~\eref{c1omcond_e}, 
or needs to be strengthened with the additional requirement that $H^2(X;\R)$
be one-dimensional.
Since the Gromov-Witten invariants of a symplectic manifold~$(X,\om)$ are invariant
under deformations of~$\om$, 
\eref{c1omcond_e} needs to hold only for some symplectic form~$\om'$ deformation
equivalent to~$\om$ and $\nu$ below~\eref{c1omcond_e} can be taken to be the maximum 
of the corresponding values over all such~$\om'$ for which \eref{c1omcond_e} holds
for some~$\la$.

\begin{cnj}\label{GWeq0_cnj}
Suppose $(X,\om)$ is a compact monotone symplectic manifold with minimal 
Chern number~$\nu$, 
$$g,N\in\Z^{\ge0} ~~\hbox{with}~  2g\!+\!N\ge3,\qquad
b_s,c_s\in\Z^{\ge0},~H_s\!\in H^{2c_s}(X)~~\hbox{for}~s\!=\!1,\ldots,N.$$
If there exists $S\!\subset\!\{1,\ldots,k\}$ such~that
\BE{GWeq0_e2}b_s\!+\!c_s<\nu~~\forall\,s\!\in\!S
\qquad\hbox{and}\qquad \sum_{s\in S}b_s>3(g\!-\!1)\!+\!N,\EE
then $\blr{\tau_{b_1}H_1,\ldots,\tau_{b_N}H_N}_{g,\be}^X=0$.
\end{cnj}

\noindent
For example,
$$\blr{\underset{N-2}{\underbrace{\tau_bH^{n-b},\ldots,\tau_bH^{n-b}}},
\cdot,\cdot}_{g,d}^{\P^n}=0
\quad \forall\,N\!\ge\!3,\,b\!=\!1,2,\ldots,n~~\hbox{with}~(N\!-\!2)b>3(g\!-\!1)\!+\!N.$$
For $g\!=\!0$ and $X\!=\!\P^1$, this statement follows from the dilaton relation 
\cite[p527]{MirSym}. 
For $n\!\ge\!2$ (which is necessary for the assumptions of Conjecture~\ref{GWeq0_cnj}
to be satisfied if $g\!\ge\!1$), $\tau_bH^{n-b}$ is not a (virtual) divisor on $\ov\fM_{g,N}(\P^n,d)$ 
and there appears to be no direct geometric reason for the vanishing above.\\

\noindent
The assumption $N\!\ge\!3$ if $g\!=\!0$ in Conjecture~\ref{GWeq0_cnj} is needed.
For example,
$$\blr{H^2,H^2}_{0,\ell}^{\P^2},\blr{\tau_2,H^2}_{0,\ell}^{\P^2}=1, \quad
\blr{\tau_1H,H^2}_{0,\ell}^{\P^2}=-1,$$
where $\ell\!\in\!H_2(\P^2;\Z)$ is the standard generator;
the constraints for all invariants above satisfy~\eref{GWeq0_e2}.
For $g\!\ge\!1$, the condition $2g\!+\!N\!\ge\!3$ is forced
by the second requirement in~\eref{GWeq0_e2}.
Both conditions in~\eref{GWeq0_e2} are needed as well. 
For example,
$$\blr{\tau_1H,H^2,H}_{0,\ell}^{\P^2},\blr{\tau_2,H^2,H}_{0,\ell}^{\P^2}
\blr{\tau_2,H^2,H,H}_{0,\ell}^{\P^2}=0,\quad
\blr{\tau_3,H,H}_{0,\ell}^{\P^2},\blr{\tau_1H,H^2,H,H}_{0,\ell}^{\P^2}=1.$$
The constraints for the first three invariants above satisfy both conditions.
The constraints for the second-to-last invariant  fail the first condition,
but satisfy the second.
The constraints for the last invariant satisfy the first condition,
but fail the second.\\

\noindent
Theorem~2 in~\cite{posgen} establishes Conjecture~\ref{GWbound_cnj} for $X\!=\!\P^n$.
Theorem~2 in~\cite{g0ci} establishes the $g\!=\!0$ case of 
this conjecture for complete intersections $X\!\subset\!\P^n$
with each $H_s$ being a power of the hyperplane class $H\!\in\!H^2(\P^n)$.
Because of the conditions on~$b_s$ in Conjecture~\ref{GWeq0_cnj}, 
the assumptions of this conjecture are never satisfied if $\nu\!=\!0,1$ 
(Calabi-Yau and borderline Fano cases) or if $\nu\!=\!2$ and $g\!\ge\!1$. 
For the same reason, its conclusion is the strongest for $X\!=\!\P^n$ 
(when the Fano index~$\nu$ is maximal relative to the dimension of~$X$,
at least in the category of K\"ahler manifolds).\\

\noindent
Theorem~2 in~\cite{g0ci} and Theorem~2 in~\cite{posgen}  are obtained 
from the mirror symmetry formulas for the equivariant multi-pointed  genus~0 
Gromov-Witten invariants of complete intersections $X\!\subset\!\P^n$ and 
for the equivariant multi-pointed  genus~$g$ Gromov-Witten invariants of~$\P^n$,
respectively, established in the two papers.
Unlike the situation with Theorem~1 in~\cite{g0ci} and Theorem~1 in~\cite{posgen}
the derivation of which from the mirror symmetry formulas requires 
a significant amount of combinatorial analysis,
Theorem~2 in~\cite{g0ci} and Theorem~2 in~\cite{posgen} are immediate consequences
of these formulas.
As explained at the end of Section~\ref{EBounds_subs}, 
it should be possible to extend the mirror symmetry formulas of \cite{g0ci,posgen}
to many other targets and thus test Conjecture~\ref{GWeq0_cnj} for~them.

\section{Proof of Theorem~\ref{TW_thm}}
\label{details_sec}

\noindent
We recall the approach of~\cite{FI} to bounding recursively defined sequences  of the~form
\BE{recdfnseq_e}
n_d=\sum_{\begin{subarray}{c}d_1+d_2=d\\ d_1,d_2\ge1\end{subarray}}\!\!\!\!
f(d_1,d_2)n_{d_1}n_{d_2} \qquad\forall\,d\!\ge\!2,\EE
with $n_1,f(d_1,d_2)\!>\!0$ in Section~\ref{FI0_subs}.
Section~\ref{FI_subs} describes the attempt in~\cite{FI} to obtain Proposition~\ref{FIexp_prp} 
and to conclude the $g\!=\!0$ case of Conjecture~\ref{P2_cnj} {\it directly} from~it;
we achieve the former in Section~\ref{expand_subs}.
Section~\ref{TW_subs} presents the observations in~\cite{TW} that are used to
deduce Theorem~\ref{TW_thm} from Propositions~\ref{FIexp_prp}.

\subsection{Lower and upper bounds}
\label{FI0_subs}

\noindent
Let $n_1,n_2,\ldots$ be a sequence of numbers satisfying
\BE{stndrec_e}
n_d=a\sum_{\begin{subarray}{c}d_1+d_2=d\\ d_1,d_2\ge1\end{subarray}}\!\!\!\!\!\!n_{d_1}n_{d_2}
\qquad\forall\,d\ge2,\EE
for some $a\!>\!0$.
The generating function
$$\Phi(q)\equiv \sum_{d=1}^{\i}n_dq^d$$
then satisfies $\Phi(q)=n_1q+a\Phi(q)^2$.
Thus,
$$\Phi(q)=\frac{1-\sqrt{1-4an_1q}}{2a}
=-\frac{1}{2a}\sum_{d=1}^{\i}\binom{1/2}{d}(-4an_1q)^d
=\sum_{d=1}^{\i}\frac{(2d\!-\!2)!}{d!(d\!-\!1)!}a^{d-1}n_1^dq^d\,;$$
the middle equality above is the Binomial Theorem.

\begin{lmm}\label{rec_lmm}
If $n_1,n_2,\ldots$ is a sequence of numbers satisfying
$$n_d=a\sum_{\begin{subarray}{c}d_1+d_2=d\\ d_1,d_2\ge0\end{subarray}}\!\!\!\!\!
\frac{f(d_1)f(d_2)}{f(d)}n_{d_1}n_{d_2}
\qquad\forall\,d\ge2,$$
for some $a\!>\!0$ and $f\!:\Z^+\!\lra\!\R$, then
$$n_d=\frac{(2d\!-\!2)!}{d!(d\!-\!1)!}
\frac{a^{d-1}}{f(d)} \big(f(1)n_1\big)^d \qquad\forall\,d\ge1.$$
\end{lmm}

\begin{proof} The sequence $\wt{n}_d\!=\!f(d)n_d$ satisfies the recursion~\eref{stndrec_e}.
\end{proof}

\noindent
For each $g\!\in\!\Z^{\ge0}$ and $d\!\in\!\Z^+$, let
$$n_{g,d}=\frac{N_{g,d}}{(3d\!-\!1\!+\!g)!}\,.$$

\begin{crl}\label{FIbnd_crl}
The numbers $n_{0,d}$ satisfy
\BE{P2asymp_e0} \frac{8}{5}\bigg(\frac{1}{27}\bigg)^dd^{-7/2}
\le n_{0,d}\le \frac{45}{16}\bigg(\frac{4}{15}\bigg)^dd^{-7/2}\,.\EE 
In particular,
$$\frac{1}{27}\le \liminf_{d\lra\i}\sqrt[d]{n_{0,d}}\le 
b_+\!\equiv\!\limsup_{d\lra\i}\sqrt[d]{n_{0,d}}\le \frac{4}{15}\,.$$
\end{crl}

\begin{proof} Let 
\begin{equation*}\begin{split}
f(d_1,d_2)
&=\frac{d_1d_2((3d_1\!-\!2)(3d_2\!-\!2)(d\!+\!2)+8(d\!-\!1))}{6(3d\!-\!3)(3d\!-\!2)(3d\!-\!1)}\\
&=\frac{d_1d_2(3d_1d_2(d\!+\!2)-2d^2)}{2(3d\!-\!3)(3d\!-\!2)(3d\!-\!1)}
\hspace{1in}\hbox{where}\quad d\equiv d_1\!+\!d_2\,.
\end{split}\end{equation*}
We note that 
\BE{FIbnd_e3}\begin{split}
\frac{1}{54}\frac{d_1d_2(3d_1\!-\!2)(3d_2\!-\!2)}{d(3d\!-\!2)}
&=\frac{d_1d_2(3d_1\!-\!2)(3d_2\!-\!2)(d\!-\!1)}{6(3d\!-\!3)(3d\!-\!2)3d}\\
&\le f(d_1,d_2)
\le  \frac{d_1d_2\cdot 3d_1d_2(d\!-\!\frac23)}{2\frac{3d}{2}(3d\!-\!2)\frac{5d}{2}}
=\frac{2}{15}\frac{d_1^2d_2^2}{d^2}  
\end{split}\EE
for all $d_1,d_2\!\in\!\Z^+$ and $d\!\equiv\!d_1\!+\!d_2$.
By  \cite[(10.4)]{RT}, 
\BE{P2g0dfn_e}
n_{0,1}=\frac12, \qquad
n_{0,d}=\sum_{\begin{subarray}{c}d_1+d_2=d\\ d_1,d_2\ge1\end{subarray}}\!\!\!\!
f(d_1,d_2)n_{0,d_1}n_{0,d_2} \quad\forall~d\!\ge\!2\,.\EE
By \eref{FIbnd_e3} and Lemma~\ref{rec_lmm},
$$ \frac92 \frac{(2d)!}{(d!)^2}\bigg(\frac{1}{108}\bigg)^dd^{-3}
\le  n_{0,d}\le  \frac{15}{4}\frac{(2d)!}{(d!)^2}\bigg(\frac{1}{15}\bigg)^dd^{-3}\,.$$
By Stirling's formula \cite[Theorem~15.19]{A},
\BE{Stirling_e}
\frac{16}{45} 4^d d^{-1/2}
\le \frac{4^d}{\sqrt{\pi d}}\bigg(1+\frac1{4d}\bigg)^{-2}
\le\frac{(2d)!}{(d!)^2}\le \frac{4^d}{\sqrt{\pi d}}\bigg(1+\frac1{8d}\bigg)
\le \frac34 4^d d^{-1/2}\,.\EE\\
Combining the last two statements, we obtain~\eref{P2asymp_e0}.
\end{proof}

\subsection{The reasoning in \cite{FI}}
\label{FI_subs}

\noindent
Proposition~\ref{FIexp_prp} and Corollary~\ref{F1_crl} below describe the behavior
the generating series
\BE{F0dfn_e}F_0(z)\equiv\frac13\sum_{d=1}^{\i}n_{0,d}e^{dz}
\quad\hbox{and}\quad
F_1(z)\equiv\sum_{d=1}^{\i}n_{1,d}e^{dz}, \qquad z\!\in\!\C,\EE
for the counts of genus~0 and~1 curves in~$\P^2$.
The statement of Proposition~\ref{FIexp_prp} appears in~\cite{FI}.
This statement, not established in~\cite{FI}, is behind the claim 
in~\cite{FI} to confirm the $g\!=\!0$ case of Conjecture~\ref{P2_cnj}.
We prove Proposition~\ref{FIexp_prp} in Section~\ref{expand_subs}.\\

\noindent
For $\de\!\in\!\R^+$ and $x_0\!\in\!\R$, let
$$B_{\de}(0)=\big\{z\!\in\!\C\!:|z|\!<\!\de\big\}, \quad
\C_{x_0}^<=\big\{z\!\in\!\C\!:\Re z\!<\!x_0\big\}, \quad
\C_{x_0}^{\le}=\big\{z\!\in\!\C\!:\Re z\!\le\!x_0\big\}.$$
We define $z^{1/2}$ on $\C\!-\!\R^+$ by the condition $\Im(z^{1/2})\!\ge\!0$.

\begin{prp}\label{FIexp_prp}
There exists $x_0\!\in\!\R$ such that the power series~$F_0$
converges on $\C_{x_0}^<$ and diverges outside of~$\C_{x_0}^{\le}$.
Furthermore, there~exist $\de\!\in\!\R^+$,  $a_0,a_2\!\in\!\R$, and
$a_{2d}\!\in\!\R$ and $a_{2d+1}\!\in\!\fI\R$
with $d\!\in\!\Z$, $d\!\ge\!2$, such~that $a_5\!\in\!\fI\R^-$ and
\BE{FIexp_e}
F_0(x_0\!+\!z)=a_0+a_2z+a_4z^2+\sum_{d=5}^{\i}a_dz^{d/2}\EE
for all $z\!\in\!B_{\de}(0)$ with $\Re(z)\!\le\!0$.
\end{prp}

\noindent
The first statement of this proposition is immediate from~\eref{P2asymp_e0}.
Furthermore, $e^{x_0}b_+\!=\!1$.\\

\noindent
Since $n_{0,d}\!\in\!\R^+$ for all~$d$, there is no neighborhood of $z\!=\!x_0$
on (all of) which this series converges; otherwise, every point $z_0$ with 
$\Re z_0\!=\!x_0$ would have such a neighborhood.
By \eref{P2g0dfn_e},
\BE{F0cond_e}(9+2F_0'-3F_0'')F_0'''=2F_0-11F_0'+18F_0''+(F_0'')^2\,;\EE
there is a sign typo in \cite[(2.55)]{FI}, which is corrected in \cite[(2.57)]{FI}.
Since $n_{0,d}\!>\!0$ for all $d\!\in\!\Z^+$,
\BE{F0cond_e2} 0<F_0(z)<F_0'(z)<F_0''(z)<F_0'''(z) \qquad \forall~z\!\in\!(-\i,x_0)\,.\EE
Combined with~\eref{F0cond_e}, this implies that 
\BE{F0bound_e}\begin{split}
&\quad 3F_0''(z)-2F_0'(z)<9 \qquad\forall\,z\!\in\!(-\i,x_0), \\
& 2F_0(x_0)-11F_0'(x_0)+18F_0''(x_0)+F_0''(x_0)^2>0.
\end{split}\EE
By~\eref{F0cond_e2} and the first statement in~\eref{F0bound_e},
the series for $F_0$, $F_0'$, and $F_0''$ converge at $z\!=\!x_0$.
Along with~\eref{F0cond_e}, this implies~that 
\BE{F0bound_e3} 3F_0''(x_0)-2F_0'(x_0)=9\,;\EE
otherwise, \eref{F0cond_e} could be used to compute all derivatives of $F_0$ at $z\!=\!x_0$
and $F_0$ could be extended on a neighborhood of~$x_0$.\\

\noindent
Since the power series for $F_0''$ converges at~$z\!=\!x_0$ and 
the power series for~$F_0'''$ does not converge,
$$\limsup_{d\lra\i} \frac{\ln n_{0,d}-d\ln{b_+}}{\ln d}\in [-4,-3].$$
According to \cite[p170]{FI},  this also implies that $F_0$ admits an expansion
around $z\!=\!x_0$ of the~form
\BE{F0exp_e} F_0(x_0\!+\!z)=c_0+c_1z+\frac{c_2z^2}{2}+\la z^{2+\al}+\ldots\EE
for some $\al\!\in\!(0,1)$.
By~\eref{F0exp_e} and~\eref{F0cond_e}, 
$$\big((9\!+\!2c_1\!-\!3c_2)-3\la(1\!+\!\al)(2\!+\!\al)z^{\al}\big)
\cdot\la\al(1\!+\!\al)(2\!+\!\al)z^{\al-1}
=2c_0-11c_1+18c_2+c_2^2+o(1).$$
Along with~\eref{F0bound_e3}, this gives
\BE{F0coeff_e} 9\!+\!2c_1\!-\!3c_2=0,\quad 2\al\!-\!1=0,\quad
-3\la\al(1\!+\!\al)^2(2\!+\!\al)^2=2c_0\!-\!11c_1\!+\!18c_2\!+\!c_2^2\,,\EE
and implies~\eref{FIexp_e}.
However, the existence of the expansion~\eref{F0exp_e} does not follow 
just from the convergence of $F_0''$ at~$x_0$ and 
the non-convergence of~$F_0'''$ there.
In Section~\ref{expand_subs}, we justify~\eref{FIexp_e} bypassing~\eref{F0exp_e}.\\

\noindent
According to \cite[p170]{FI}, \eref{FIexp_e} {\it corresponds to}  
the $g\!=\!0$ case of Conjecture~\ref{P2_cnj}.
However, \eref{FIexp_e} by itself can  describe {\it at most} a suitable $\limsup$.
It does not imply even the genus~0 case of Corollary~\ref{TWexp_crl}.
For example, replacing the numbers~$n_{0;d}$ in~\eref{F0dfn_e}  by the numbers
$$n_{0,d}'=\begin{cases}n_{0,d/2},&\hbox{if}~d\!\in\!2\Z;\\
0,&\hbox{if}~d\!\not\in\!2\Z;\end{cases}$$
would break the validity of the first equality in~\eref{TWexpcrl_e}
without affecting the validity of the conclusion of Proposition~\ref{FIexp_prp}.\\

\noindent
{\it Mathematica} suggests that the numbers on the left-hand side of~\eref{TWexpcrl_e}
are increasing (after the first few terms), but it is not clear how this can be proved.
In light of Kontsevich's recursion~\eref{P2g0dfn_e}, this could be a special case
of Conjecture~\ref{asympgrowth_cnj}.
Combined with the conclusion of Conjecture~\ref{asympgrowth_cnj} 
for the  numbers~$n_{0,d}$ given by~\eref{P2g0dfn_e},
Proposition~\ref{FIexp_prp} 
would at least imply  the first statement of Corollary~\ref{TWexp_crl}.

\begin{crl}\label{F1_crl}
Let $x_0\!\in\!\R^+$ be as in Proposition~\ref{FIexp_prp}.
The power series~$F_1$ converges on $\C_{x_0}^<$ and diverges outside of~$\C_{x_0}^{\le}$.
Furthermore, there~exist $\de\!\in\!\R^+$ and 
$b_{2d}\!\in\!\R$ and $b_{2d+1}\!\in\!\fI\R$
with $d\!\in\!\Z$, $d\!\ge\!-1$, such~that
\BE{F1exp_e}F_1'(x_0\!+\!z)=
-\frac{1}{48z}+\sum_{d=-1}^{\i}\!b_dz^{d/2}\EE
for all $z\!\in\!B_{\de}(0)$ with $\Re(z)\!\le\!0$.
\end{crl}

\begin{proof} By \cite[(8)]{Pand99},
\BE{P2g1dfn_e}
n_{1,d}=\frac{(d\!-\!1)(d\!-\!2)}{216}n_{0,d}+\frac{1}{27d}
\sum_{\begin{subarray}{c}d_0+d_1=d\\ d_0,d_1\ge1\end{subarray}}\!\!\!\!
(3d_0^2\!-\!2d_0)d_1\, n_{0,d_0}n_{1,d_1}\,.\EE
This implies that 
\BE{F1cond_e} (9+2F_0'-3F_0'')F_1'=
\frac18\big(F_0'''-3F_0''+2F_0').\EE
By~\eref{F1cond_e} and the first statement of Proposition~\ref{FIexp_prp},  
the series~$F_1$ converges on $\C_{x_0}^<$.\\

\noindent
By~\eref{FIexp_e} and~\eref{F0bound_e3}, 
\BE{FOvan_e}\begin{split}
9+2F_0'(x_0\!+\!z)-3F_0''(x_0\!+\!z)&=
z^{1/2}\sum_{d=0}^{\i}\frac{d\!+\!3}{4}\big(4a_{d+3}-3(d\!+\!5)a_{d+5}\big)z^{d/2},\\
F_0'''(x_0\!+\!z)&=z^{-1/2}\sum_{d=0}^{\i}\frac{(d\!+\!1)(d\!+\!3)(d\!+\!5)}{8}a_{d+5}z^{d/2},
\end{split}\EE
with $a_3\!\equiv\!0$.
Along with~\eref{F1cond_e}, this implies~\eref{F1exp_e}.
\end{proof}

\noindent
Since the power series~$F_1$ converges for $z\!\in\!\C_{x_0}^<$,
$$\limsup_{d\lra\i}\sqrt[d]{n_{1,d}}
\le e^{-x_0}=\limsup_{d\lra\i}\sqrt[d]{n_{0,d}}\,.$$
The opposite inequality follows directly from \eref{P2g1dfn_e}; it also holds for $\liminf$.
Thus, the existence of $b\!\in\!\R^+$ such that the first equality in~\eref{TWexpcrl_e} holds
implies that the second equality also holds.

\subsection{The observations in \cite{TW}}
\label{TW_subs}

\noindent
We now describe the statements in~\cite{TW} that have completed 
the proof of the $g\!=\!0$ case of Conjecture~\ref{P2_cnj}, 
initiated in~\cite{FI} and continued in~\cite{AsympGWnotes},
and have extended it to the $g\!=\!1$ case.\\

\noindent
Since the functions~\eref{F0dfn_e} are $2\pi\fI$-periodic, 
they do not extend analytically over a neighborhood of \hbox{$x_0\!+\!2\pi k\fI$} 
for any $k\!\in\!\Z$.
By Lemma~\ref{TW_lmm0} below,  they extend analytically around all other points of
the vertical line $\Re z\!=\!x_0$ in~$\C$.
The two asymptotic expansions of Theorem~\ref{TW_thm} then follow from 
Lemma~\ref{TW_lmm},  Proposition~\ref{FIexp_prp}, and Corollary~\ref{F1_crl}.

\begin{lmm}[{\cite[Lemma~3.1]{TW}}]\label{TW_lmm0}
Let $F_0$ and~$F_1$ be as in~\eref{F0dfn_e} and
$x_0$ be as in Proposition~\ref{FIexp_prp}.
Then
\BE{TWlmm0_e} 3F_0''\big(x_0\!+\!\fI y_0\big)-2F_0'\big(x_0\!+\!\fI y_0\big)\neq9 
\qquad\forall~y_0\!\in\!\R\!-\!2\pi\Z\,.\EE
Thus, the functions~$F_0$ and $F_1$ extend analytically over a neighborhood of 
every point $x_0\!+\!\fI y_0$ with \hbox{$y_0\!\in\!\R\!-\!2\pi\Z$}.
\end{lmm}

\begin{proof}
The first statement follows from $n_{0,d}\!>\!0$ and~\eref{F0bound_e3}, since
\begin{equation*}\begin{split}
\Re\Big(3F_0''\big(x_0\!+\!\fI y_0\big)-2F_0'\big(x_0\!+\!\fI y_0\big)\Big)
&=\sum_{d=1}^{\i}\big(3d\!-\!2)d\,n_{0,d}e^{dx_0}\cos(dy_0)\\
&<\sum_{d=1}^{\i}\big(3d\!-\!2)d\,n_{0,d}e^{dx_0}
=3F_0''(x_0)-2F_0'(x_0)=9\,.
\end{split}\end{equation*}
By~\eref{TWlmm0_e}, \eref{F0cond_e} and~\eref{F1cond_e} can used to compute 
all derivatives of $F_0$ and $F_1$ at $z_0\!\equiv\!x_0\!+\!\fI y_0$ and 
to extend $F_0$ and~$F_1$ around~$z_0$.
\end{proof}

\noindent
The remaining part of \cite[Lemma~3.1]{TW} is equivalent to Proposition~\ref{FIexp_prp},
established 5~years earlier in~\cite{AsympGWnotes}.
The proof in~\cite{TW} provides an alternative argument for Proposition~\ref{FIexp_prp};
this argument is somewhat shorter in length than Section~\ref{expand_subs},
but is more ad hoc and does not include recursions for the coefficients~$a_d$ 
in~\eref{FIexp_e}.\\

\noindent 
The crucial observation of~\cite{TW} is that the asymptotic behavior of
the coefficients~$n_d$ of a generating series~$F$ as in~\eref{TWlmm_e0} below
is described by expansions around~$x_0$ {\it if} $F$ has no additional singular 
points on the vertical line $\Re z\!=\!x_0$.
This observation is reformulated in greater generality as Lemma~\ref{TW_lmm} below,
which appears similar to some asymptotic analysis statements in combinatorics.
Two special cases of this lemma are considered in~\cite{TW} and treated 
separately, but the argument in the second case in~\cite{TW}
applies to the general case of Lemma~\ref{TW_lmm} without any changes of substance.\\

\noindent
For $k\!\in\!\R^+$,  let
$$\Ga(k)\equiv \int_0^{\i}\!t^{k-1}e^{-t}\d t$$
denote the value of the $\Ga$ function at $k$.

\begin{lmm}\label{TW_lmm}
Let $x_0\!\in\!\R$, $\de_0\!\in\!\R^+$,
$n_d\!\in\!\C$ for  $d\!\in\!\Z^+$, and $a_d\!\in\!\C$ for $2d\!\in\!\Z$.
If the power series
\BE{TWlmm_e0} F(z) \equiv \sum_{d=1}^{\i}n_d e^{dz}, \qquad z\in\C,\EE
converges on~$\C_{x_0}^<$, 
extends analytically over a neighborhood of  every point 
$x_0\!+\!\fI y_0$ with $y_0\!\in\!\R\!-\!2\pi\Z$, and satisfies
$$F(x_0\!+\!z)=\sum_{\begin{subarray}{c}2k\in \Z\\ -1\le k \end{subarray}}
\!\!a_kz^k$$
for all $z\!\in\!B_{\de_0}(0)$ with $\Re(z)\!<\!0$,
then for each $N\!\in\!\Z$ there exists $C_N\!\in\!\R$ such~that
\BE{TWlmm_e} 
\bigg|n_d -e^{-dx_0}\bigg(\!\!-a_{-1}+\frac{1}{\pi\fI} \!\!\!\!\!
\sum_{\begin{subarray}{c}2k\in \Z-2\Z\\ -1< k<N-1\end{subarray}}
\!\!\!\!\!\!\!\!\!a_k \Ga\big(k\!+\!1\big)d^{-k-1}\bigg)\bigg|\le 
C_N e^{-dx_0} d^{-N-\frac12} \quad\forall\,d\!\in\!\Z^+\,.\EE  
\end{lmm}

\begin{proof}[{\it Proof {\cite[pp8-11]{TW}}}]
For $d\!\in\!\Z^+$ and $k,\de\!\in\!\R^+$,  let
\BE{TWlmm_e1a}\Ga_{d;\de}(k)\equiv\int_0^{\de}\!t^{k-1}e^{-dt}\d t
=d^{-k}\!\!\int_0^{d\de}\!t^{k-1}e^{-t}\d t< d^{-k}\Ga(k)\,.\EE
If $0\!<\!N\!\le\!k$, then
\BE{TWlmm_e1c} 0< \Ga_{d;\de}(k) \le \de^{k-N}\Ga_{d;\de}(N)
< \de^{k-N}d^{-N}\Ga(N).\EE
If $N\!\ge\!k$, then
\BE{TWlmm_e1d}
0<\Ga(k)-d^k\Ga_{d;\de}(k)=\int_{d\de}^{\i}\!t^{k-1}e^{-t}\d t
\le (d\de)^{k-N}\!\!\!\int_{d\de}^{\i}\!t^{N-1}e^{-t}\d t
< (d\de)^{k-N}\Ga(N).\EE

\vspace{.2in}

\noindent
The assumptions on~$F$ imply that there exists $\de\!\in\!(0,\de_0)$ so that $F$ 
extends analytically over the~region 
$$\big\{x\!+\!\fI y\!:x\!\in\![x_0,x_0\!+\!2\de],~y\!\in\![0,2\pi]\big\}
-\big\{x_0,x_0\!+\!2\pi\fI\big\}\subset \C$$
with
\BE{TWlmm_e1}\begin{aligned}
F\big(x_0\!+\!re^{\fI\th}\big)&=
\sum_{\begin{subarray}{c}2k\in \Z\\ -1\le k \end{subarray}}\!\!a_k r^ke^{\fI k\th}
&\qquad &\forall~(r,\th)\!\in\!(0,2\de)\!\times\![0,\pi],\\
F\big(x_0\!+\!2\pi\fI\!+\!re^{\fI\th}\big)&=
\sum_{\begin{subarray}{c}2k\in \Z\\ -1\le k \end{subarray}}\!\!a_k r^ke^{\fI k\th}
&\qquad &\forall~(r,\th)\!\in\!(0,2\de)\!\times\![\pi,2\pi].
\end{aligned}\EE
For $\ep\!\in\!(0,\de)$, define oriented curves in $\C$ by 
\begin{alignat*}{3}
\cC_{\ep;-}^{\v}&=\big\{x_0\!-\!\ep\!+\!\fI t\!:t\!\in\![0,2\pi]\big\}, &~~
\cC_{\ep;-}^{\h}&=\big\{x_0\!+\!t\!:t\!\in\![\ep,\de]\big\}, &~~
\cC_{\ep;-}^{\circ}&=\big\{x_0\!+\!\ep e^{\fI t}\!:t\!\in\![0,\pi]\big\}, \\
\cC_+^{\v}&=\big\{x_0\!+\!\de\!+\!\fI t\!:t\!\in\![0,2\pi]\big\}, &~~
\cC_{\ep;+}^{\h}&=\big\{x_0\!+\!t\!+\!2\pi\fI\!:t\!\in\![\ep,\de]\big\},&~~
\cC_{\ep;+}^{\circ}&=\big\{x_0\!+\!2\pi\fI\!+\!\ep e^{\fI t}\!:t\!\in\![\pi,2\pi]\big\};
\end{alignat*}
see Figure~\ref{TW_fig}.
By~\eref{TWlmm_e1}, there exist $C_{\de}^h,C_{\de}^{\circ}\!\in\!\R$ such that 
\begin{gather}
\label{TWlmm_e2a}
\sum_{\begin{subarray}{c}2k\in \Z-2\Z\\ -1< k \end{subarray}}
\!\!\!\!\!\!\!|a_k|t^k\le C_{\de}^ht^{-1/2}
\quad\forall~t\!\in\!(0,\de],\\
\label{TWlmm_e2b}
\big|F(z)e^{-dz}-a_{-1}\big(z\!-\!x_0\!-\!(1\!\pm\!1)\pi\fI\big)^{-1}e^{-dx_0}\big|\le 
C_{\de}^{\circ}e^{-dx_0}\ep^{-1/2}
\quad\forall~z\!\in\!\cC_{\ep;\pm}^{\circ},\,\ep\!\in\!(0,\de),\,d\!\in\!\Z^+.
\end{gather}

\vspace{.2in}

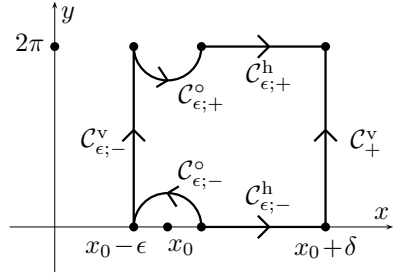
\begin{figure}
\begin{pspicture}(0,-.5)(10,3.2)
\psset{unit=.3cm}
\psline[linewidth=.03]{->}(18,0)(35,0)\psline[linewidth=.03]{->}(20,-2)(20,10)
\psarc[linewidth=.1](25,0){1.5}{0}{180}\psarc[linewidth=.1](25,8){1.5}{180}{360}
\psline[linewidth=.1](23.5,0)(23.5,8)\psline[linewidth=.1](32,0)(32,8)
\psline[linewidth=.1](32,0)(26.5,0)\psline[linewidth=.1](32,8)(26.5,8)
\pscircle*(25,0){.2}\rput(25.6,-.9){\sm{$x_0$}}
\pscircle*(20,8){.2}\rput(18.9,8.2){\sm{$2\pi$}}
\pscircle*(23.5,0){.2}\rput(22.7,-.9){\sm{$x_0\!-\!\ep$}}
\pscircle*(32,0){.2}\rput(32,-.9){\sm{$x_0\!+\!\de$}}
\pscircle*(26.5,0){.2}\pscircle*(23.5,8){.2}\pscircle*(26.5,8){.2}\pscircle*(32,8){.2}
\psline[linewidth=.1](29.5,0)(29,-.5)\psline[linewidth=.1](29.5,0)(29,.5)
\psline[linewidth=.1](29.5,8)(29,7.5)\psline[linewidth=.1](29.5,8)(29,8.5)
\psline[linewidth=.1](23.5,4.3)(23,3.8)\psline[linewidth=.1](23.5,4.3)(24,3.8)
\psline[linewidth=.1](32,4.3)(31.5,3.8)\psline[linewidth=.1](32,4.3)(32.5,3.8)
\psline[linewidth=.1](25,1.5)(25.5,1.8)\psline[linewidth=.1](25,1.5)(25.4,.9)
\psline[linewidth=.1](25,6.5)(24.5,6.2)\psline[linewidth=.1](25,6.5)(24.6,7.1)
\rput(29.5,1.3){\sm{$\cC_{\ep;-}^{\h}$}}\rput(29.5,6.7){\sm{$\cC_{\ep;+}^{\h}$}}
\rput(22.2,3.7){\sm{$\cC_{\ep;-}^{\v}$}}\rput(33.8,3.7){\sm{$\cC_+^{\v}$}}
\rput(26.5,2.3){\sm{$\cC_{\ep;-}^{\circ}$}}\rput(26.5,5.7){\sm{$\cC_{\ep;+}^{\circ}$}}
\rput(34.5,.6){\sm{$x$}}\rput(20.6,9.4){\sm{$y$}}
\end{pspicture}
\caption{The curves in the proof of Lemma~\ref{TW_lmm}.}
\label{TW_fig}
\end{figure} 

\noindent
For each $d\!\in\!\Z^+$,
\BE{TWlmm_e3}\begin{split}
2\pi\fI\, n_d= \int_{\cC_{\ep;-}^{\v}}\!\!\!\!\!\!\!\!F(z)e^{-dz}\d z
= \int_{\cC_+^{\v}}\!\!\!\!\!\!F(z)e^{-dz}\d z&-
\bigg(\int_{\cC_{\ep;-}^{\circ}}\!\!\!\!\!\!\!\!F(z)e^{-dz}\d z 
+\int_{\cC_{\ep;+}^{\circ}}\!\!\!\!\!\!\!F(z)e^{-dz}\d z\!\!\bigg)\\
&+\bigg(\int_{\cC_{\ep;-}^h}\!\!\!\!\!\!\!\!F(z)e^{-dz}\d z 
-\int_{\cC_{\ep;+}^h}\!\!\!\!\!\!\!F(z)e^{-dz}\d z\!\!\bigg).
\end{split}\EE
By the compactness of $\cC_+^{\v}$, there exists $C_{\de}^{\v}\!\in\!\R$ such that 
\BE{TWlmm_e5}
\bigg|\int_{\cC_+^{\v}}\!\!\!\!\!F(z)e^{-dz}\d z\bigg|
\le C_{\de}^{\v}\,e^{-d(x_0+\de)} \qquad\forall\,d\!\in\!\Z^+\,.\EE
By~\eref{TWlmm_e2b}, 
\BE{TWlmm_e7}
\bigg|\bigg(\int_{\cC_{\ep;-}^{\circ}}\!\!\!\!\!\!\!\!F(z)e^{-dz}\d z+
\int_{\cC_{\ep;+}^{\circ}}\!\!\!\!\!\!\!\!F(z)e^{-dz}\d z\bigg)
-2\pi\fI\, a_{-1}e^{-dx_0}\bigg| \le 2\pi C_{\de}^{\circ}e^{-dx_0}\ep^{1/2}
~~\forall\,\ep\!\in\!(0,\de),\,d\!\in\!\Z^+\,.\EE
By~\eref{TWlmm_e1}, 
\begin{equation*}\begin{split}
\int_{\cC_{\ep;-}^h}\!\!\!\!\!\!\!\!F(z)e^{-dz}\d z 
-\int_{\cC_{\ep;+}^h}\!\!\!\!\!\!\!F(z)e^{-dz}\d z
&=2\!\!\!\!\!\!\sum_{\begin{subarray}{c}2k\in \Z-2\Z\\ -1<k \end{subarray}}\!\!\!\!
a_k 
\!\!\int_{\ep}^{\de}\!\!t^ke^{-d(x_0+t)}\d t\\
&=2e^{-dx_0}\!\!\!\!\!\!\sum_{\begin{subarray}{c}2k\in \Z-2\Z\\ -1\le k \end{subarray}}
\!\!\!\!\!\!a_k
\Ga_{d;\de}(k\!+\!1) 
-2e^{-dx_0}\!\!\!\!\!\!\sum_{\begin{subarray}{c}2k\in \Z-2\Z\\ -1\le k \end{subarray}}\!\!\!\!a_k 
\!\!\int_0^{\ep}\!\!t^ke^{-dt}\d t\,.
\end{split}\end{equation*}
Along with~\eref{TWlmm_e2a}, this gives
\BE{TWlmm_e9}\begin{split}
\bigg|\bigg(\int_{\cC_{\ep;-}^h}\!\!\!\!\!\!\!\!F(z)e^{-dz}\d z 
-\int_{\cC_{\ep;+}^h}\!\!\!\!\!\!\!F(z)e^{-dz}\d z\!\!\bigg)
-2e^{-dx_0}\!\!\!\!\!\!\sum_{\begin{subarray}{c}2k\in \Z-2\Z\\ -1\le k \end{subarray}}
\!\!\!\!\!\!\!a_k\Ga_{d;\de}(k\!+\!1)\bigg|\le 4C_{\de}^he^{-dx_0}\ep^{1/2}.
\end{split}\EE
By~ \eref{TWlmm_e1d}, \eref{TWlmm_e1c}, and~\eref{TWlmm_e2a}, 
\BE{TWlmm_e11}\begin{split}
&\bigg|\sum_{\begin{subarray}{c}2k\in \Z-2\Z\\ -1\le k \end{subarray}}
\!\!\!\!\!\!\!a_k\Ga_{d;\de}(k\!+\!1) -
\sum_{\begin{subarray}{c}2k\in \Z-2\Z\\ -1< k<N-1\end{subarray}}
\!\!\!\!\!\!\!\!\!a_k \Ga\big(k\!+\!1\big)d^{-k-1}  \bigg|\\
&\qquad \le 
\sum_{\begin{subarray}{c}2k\in \Z-2\Z\\ -1< k<N-1\end{subarray}}
\!\!\!\!\!\!\!\!\!|a_k|\big|\Ga_{d;\de}\big(k\!+\!1\big)-\Ga\big(k\!+\!1\big)d^{-k-1}\big|
+\sum_{\begin{subarray}{c}2k\in \Z-2\Z\\ N-1< k\end{subarray}}
\!\!\!\!\!\!|a_k|\Ga_{d;\de}\big(k\!+\!1\big)\\
&\qquad \le 
\sum_{\begin{subarray}{c}2k\in \Z-2\Z\\ -1< k<N-1\end{subarray}}
\!\!\!\!\!\!\!\!\!|a_k|(d\de)^{k-N+\frac12}\Ga\big(N\!+\!1/2\big)d^{-k-1}
+\sum_{\begin{subarray}{c}2k\in \Z-2\Z\\ N-1< k\end{subarray}}
\!\!\!\!\!\!|a_k|\de^{k-N+\frac12}\Ga\big(N\!+\!1/2\big)d^{-N-\frac12}\\
&\qquad =\Ga\big(N\!+\!1/2\big)(d\de)^{-N-\frac12} \!\!\!
\sum_{\begin{subarray}{c}2k\in \Z-2\Z\\ -1< k\end{subarray}}
\!\!\!\!\!\!|a_k|\de^{k+1} \le C_{\de}^h\de^{1/2}\Ga\big(N\!+\!1/2\big)(d\de)^{-N-\frac12}.
\end{split}\EE
By~\eref{TWlmm_e3}-\eref{TWlmm_e11}, there exists $C_N\!\in\!\R$ such that 
$$\bigg|n_d+e^{-dx_0}\bigg(
a_{-1} -\frac{1}{\pi\fI} \!\!\!\!\!
\sum_{\begin{subarray}{c}2k\in \Z-2\Z\\ -1< k<N-1\end{subarray}}
\!\!\!\!\!\!\!\!\!a_k \Ga\big(k\!+\!1\big)d^{-k-1}\bigg)\bigg|
\le C_Ne^{-dx_0}\big(e^{-d\de}\!+\!\ep^{1/2}\!+\!d^{-N-\frac12}\big)$$
for all $\ep\!\in\!(0,\de)$ and $d\!\in\!\Z^+$.
Sending $\ep$ to~0, we obtain the claim.
\end{proof}

\begin{rmk}\label{TW_rmk}
Lemma~\ref{TW_lmm} can be used to obtain asymptotics similar to~\eref{TWlmm_e} 
for power series~$F(z)$ as in its statement satisfying 
\BE{TWrmk_e}F(x_0\!+\!z)=\sum_{\begin{subarray}{c}2k\in \Z\\ k_0\le k \end{subarray}}
\!\!a_kz^k\EE
for some $k_0\!\in\!\Z^-$.
The  coefficients~$a_k$ with $k\!\in\!\Z^-$ can be eliminated by 
adding appropriate multiples of the power series
$$F_k(z)=\sum_{d=1}^{\i}d^{1-k}e^{-dx_0}e^{dz}
=\bigg\{\frac{\d}{\d z}\bigg\}^{1-k}\bigg(\frac{e^{z-x_0}}{1-e^{z-x_0}}\bigg)\,.$$
The  coefficients~$a_k$ with $k\!\in\!\Z^-\!-\!2\Z$ can then be eliminated by 
integrating~$F(z)$ enough times.
These two modifications do not break the remaining requirements imposed on~$F$ by Lemma~\ref{TW_lmm} 
and reduce any expansion as in~\eref{TWrmk_e} to one with $k_0\!=\!0$.
In such a case, the $\ep\!=\!0$ contour of \cite[pp8,9]{TW} suffices.
\end{rmk}

\subsection{Proof of Proposition~\ref{FIexp_prp}}
\label{expand_subs}

\noindent
It remains to establish the last statement of Proposition~\ref{FIexp_prp}.

\begin{lmm}\label{F0uniq_lmm}
Let $x_0\!\in\!\R$, $\de\!\in\!\R^+$, $x^*\!\in\!(x_0\!-\!\de,x_0)$, and 
\BE{F0uniqlmm_e0}F,G\!:(x_0\!-\!\de,x_0)\lra \R\EE
be solutions of~\eref{F0cond_e} satisfying~\eref{F0cond_e2} and 
\BE{F0uniqlmm_e}  F(x^*)=G(x^*), \qquad F'(x^*)=G'(x^*), \qquad F''(x^*)<G''(x^*)\,.\EE
Then $F''(x)\!<\!G''(x)$ for all $x\!\in\![x^*,x_0]$.
\end{lmm}

\begin{proof} We can assume that $\de\!<\!1$.
Suppose $x'\!\in\!(x^*,x_0)$ and $F'''(x)\!<\!G'''(x)$ for all $x\!\in\!(x^*,x')$.
Then,
$$0\le G'(x')-F'(x')=\int_{x^*}^{x'}\!\!\big(G''(x)\!-\!F''(x)\big)\d x
\le \big(G''(x')\!-\!F''(x')\big)(x'\!-\!x^*)\le G''(x')\!-\!F''(x').$$
Thus, 
$$F(x')\le G(x'),~F'(x')\le G'(x'),~F''(x')<G''(x'),~F''(x')\!-\!F'(x')\le G''(x')\!-\!G'(x').$$
Along with~\eref{F0cond_e}, this implies that $F'''(x')\!<\!G'''(x')$.
The claim now follows.
\end{proof}

\begin{crl}\label{F0uniq_crl}
Let $F,G$ be solutions of~\eref{F0cond_e} as in~\eref{F0uniqlmm_e0} satisfying~\eref{F0cond_e2}
and $x'\!\in\!(x^*,x_0)$ be such~that 
\BE{F0uniqcrl_e} F(x')=G(x^*), \quad F'(x')=G'(x^*), \quad
\lim_{x\lra^-x_0}\!\!\!\!\!F'''(x),\lim_{x\lra^-x_0}\!\!\!\!\!G'''(x)=\i.\EE
Then $F''(x')\!\ge\!G''(x^*)$.
\end{crl}

\begin{proof} Suppose $F''(x')\!<\!G''(x^*)$.
Let $y\!=\!x'\!-\!x^*\!>\!0$.
Since \eref{F0cond_e} is a homogeneous differential equation,
the~function 
$$\wt{F}\!:\big(x_0\!-\!y\!-\!\de,x_0\!-\!y\big)\lra\R, \qquad
\wt{F}(x)=F(x\!+\!y),  $$
is a solution of~\eref{F0cond_e} satisfying~\eref{F0cond_e2} and 
$$\wt{F}(x^*)=G(x^*), \quad \wt{F}'(x^*)=G'(x^*)\,,\quad
\wt{F}''(x^*)<G''(x^*), \quad
\lim_{x\lra^-x_0-y}\!\!\!\!\!\wt{F}'''(x)=\i\,.$$
This contradicts the conclusion of Lemma~\ref{F0uniq_lmm}, since $G'''(x_0\!-\!y)$ is finite.
\end{proof}

\begin{lmm}\label{Frobexp_lmm}
Let $a_0,a_2\in\!\R$ be such~that 
\BE{a5cond_e0} 4a_2^2\!+\!45a_2\!+\!18a_0\!+\!567> 0.\EE
Then there exist $\de\!\in\!\R^+$ and unique 
\BE{a5cond_e} a_4\!\in\!\R,\quad a_5\!\in\!\fI\R^-,
\quad a_{2d}\!\in\!\R,~a_{2d+1}\!\in\!\fI\R~~\forall~d\!\in\!\Z,\,d\!\ge\!3,\EE
such~that the power series in \eref{FIexp_e} converges uniformly 
for $z\!\in\!B_{\de}(0)$ with $\Re(z)\!\le\!0$
to a solution of~\eref{F0cond_e}.
The number $\de\!=\!\de(a_0,a_2)$ can be chosen to depend continuously on $(a_0,a_2)\!\in\!\R^2$
satisfying~\eref{a5cond_e0}.
\end{lmm}

\begin{proof}
For arbitrary $a_d\!\in\!\C$, let
\BE{F0expass_e} F_0(x_0\!+\!z)=\sum_{d=0}^{\i}a_dz^{d/2}.\EE
The differential equation \eref{F0cond_e} is then equivalent~to 
\begin{gather*}
a_1,a_3=0\,, \qquad (9\!+\!2a_2\!-\!6a_4)a_5=0\,, \\
\begin{split}
&(9\!+\!2a_2\!-\!6a_4)\frac{(d\!+\!2)(d\!+\!4)(d\!+\!6)}{8}a_{d+6}
=2a_d-\frac{11(d\!+\!2)}{2}a_{d+2}
+\frac{9(d\!+\!2)(d\!+\!4)}{2}a_{d+4}\\
&\hspace{1.5in}-\sum_{\begin{subarray}{c}d_1+d_2=d\\ d_1,d_2\ge0\end{subarray}}\!\!\!
\frac{(d_1\!+\!1)(d_1\!+\!3)(d_1\!+\!5)(d_2\!+\!3)}{32}a_{d_1+5}
\big(4a_{d_2+3}\!-\!3(d_2\!+\!5)a_{d_2+5}\big)\\
&\hspace{1.5in}+\sum_{\begin{subarray}{c}d_1+d_2=d\\ d_1,d_2\ge0\end{subarray}}\!\!\!
\frac{(d_1\!+\!2)(d_1\!+\!4)(d_2\!+\!2)(d_2\!+\!4)}{16}a_{d_1+4}a_{d_2+4}\,.
\end{split}\end{gather*}
The last equation holds for all $d\!\ge\!0$.\\

\noindent 
If $a_5\!\neq\!0$, the last two conditions above are equivalent~to
\begin{gather}\label{arec_e1}
9\!+\!2a_2\!-\!6a_4=0\,,\qquad
-\frac{675}{32}a_5^2=2a_0-11a_2+36a_4+4a_4^2\,,\\
\label{arec_e2}
\begin{split}
&\frac{45(d\!+\!2)(d\!+\!3)(d\!+\!5)}{32}a_5a_{d+5}
=-2a_d+\frac{11(d\!+\!2)}{2}a_{d+2}
+\frac{(d\!+\!2)(d\!+\!4)}{2}\big((d\!-\!2)a_4\!-\!9\big)a_{d+4}\\
&\hspace{1.5in}-\sum_{\begin{subarray}{c}d_1+d_2=d\\ d_1,d_2\ge1\end{subarray}}\!\!\!
\frac{3(d\!+\!2)(d_1\!+\!3)(d_1\!+\!5)(d_2\!+\!3)(d_2\!+\!5)}{64}a_{d_1+5}a_{d_2+5}\\
&\hspace{1.5in}+\sum_{\begin{subarray}{c}d_1+d_2=d\\ d_1,d_2\ge1\end{subarray}}\!\!\!
\frac{(d_1^2\!+\!d_2^2\!-\!d_1d_2\!-\!4)(d_1\!+\!4)(d_2\!+\!4)}{16}a_{d_1+4}a_{d_2+4}\,;
\end{split}\end{gather}
the last equation is valid for $d\!\ge\!1$.
By~\eref{arec_e1}, 
$$a_4=\frac32+\frac13a_2, \qquad -\frac{6075}{32}a_5^2=4a_2^2\!+\!45a_2\!+\!18a_0\!+\!567\,.$$
By~\eref{a5cond_e0}, these equations determine $a_d\!\in\!\C$ with $d\!\ge\!4$.
By induction, the coefficients~$a_d$ satisfy~\eref{a5cond_e}.\\

\noindent
It remains to show that the power series~\eref{F0expass_e} with $a_d$ given by~\eref{arec_e2} 
converges uniformly on a small disk.
Let $\wt{a}_d$ be the sequence recursively defined~by
\begin{gather*}
\wt{a}_d=1+\sum_{k=1}^6|a_k| \qquad\forall~d\le6,\\
\begin{split}
d^3\,\wt{a}_5\wt{a}_{d+5}&= 2\wt{a}_d+20d\,\wt{a}_{d+2}+
\big(64\!+\!16d a_4\big)d^2\,\wt{a}_{d+4}\\
&\quad +36\!\!\!\!\!\sum_{\begin{subarray}{c}d_1+d_2=d\\ d_1,d_2\ge1\end{subarray}}\!\!\!\!\!
d_1^3d_2^3\,\wt{a}_{d_1+5}\wt{a}_{d_2+5}
+5\!\!\!\!\!\sum_{\begin{subarray}{c}d_1+d_2=d\\ d_1,d_2\ge1\end{subarray}}\!\!\!
\!\!d_1^3d_2^3\,\wt{a}_{d_1+4}\wt{a}_{d_2+4} \quad\forall~d\ge2.
\end{split}\end{gather*}
Let $n_d$ be the sequence recursively defined~by 
$$n_1=\wt{a}_6, \qquad 
n_d=150\big(1\!+\!|a_5|^{-1}\big)\!\!\!\!\!\!
\sum_{\begin{subarray}{c}d_1+d_2=d\\ d_1,d_2\ge1\end{subarray}}\!\!\!
\frac{d_1^3d_2^3}{d^3}n_{d_1}n_{d_2} \quad\forall~ d\ge2.$$
By~\eref{arec_e2} and the sequence $\wt{a}_d$ being positive and non-decreasing,
$$|a_d|\le \wt{a}_d \le n_{d-5} \qquad\forall~d\!\ge\!6\,.$$
By Lemma~\ref{rec_lmm} and~\eref{Stirling_e}, there thus exists $C\!\in\!\R^+$ such that   
$$|a_d|\le C^d\big(1+|a_5|^{-1}\big)^d \qquad\forall\,d\!\in\!\Z^+\,.$$
It follows that \eref{F0expass_e} defines a  solution of~\eref{F0cond_e} 
around $z\!=\!x_0$ for any choice of $a_0$ and $a_2$ such 
that \eref{a5cond_e0} is satisfied.
\end{proof}

\noindent
We define
\BE{Wdfn_e}W=\big\{(a_0,a_2)\!\in\!\R^2\!: 0\!<\!a_0\!<\!a_2\!<\!9\big\}\,.\EE
For each $(a_0,a_2)\!\in\!W$, let 
$$F_{a_0,a_2}\!: \big\{z\!\in\!\C^{\le}_{x_0}\!:|x_0\!-\!z|\!<\!\de(a_0,a_2)\big\}\lra\C$$ 
be the solution~\eref{F0expass_e} of~\eref{F0cond_e} provided by Lemma~\ref{Frobexp_lmm}.
In particular, 
\BE{Flim_e}
F_{a_0,a_2}(x_0)\!=\!a_0 < F_{a_0,a_2}'(x_0)\!=\!a_2 <
F_{a_0,a_2}''(x_0)\!=\!3\!+\!\frac23a_2\,.\EE
Reducing the continuous function $\de\!=\!\de(a_0,a_2)$ if necessary, we can 
assume~that 
$$0<F_{a_0,a_2}(z)<F_{a_0,a_2}'(z)<F_{a_0,a_2}''(z)<F_{a_0,a_2}'''(z) 
\qquad \forall~z\!\in\!\big(x_0\!-\!\de(a_0,a_2),x_0\big).$$
We define
\begin{gather*}
\Phi\!:\wt{W}\!\equiv\!\{(z,a_0,a_2)\!\in\!(-\i,x_0]\!\times\!W\!: 
x_0\!-\!z\!<\!\de(a_0,a_2)\big\}\lra (-\i,x_0]\!\times\!\R^2, \\
\Phi(z,a_0,a_2)= \big(z,F_{a_0,a_2}(z),F_{a_0,a_2}'(z)\big)\,.
\end{gather*}

\begin{lmm}\label{homeom_lmm} Let $(a_0,a_2)\!\in\!W$. 
There exist neighborhoods  $W_{(a_0,a_2)}$ and $V_{(a_0,a_2)}$
of $(x_0,a_0,a_2)$ in \hbox{$(-\i,x_0]\!\times\!\R^2$} such~that the~map
$$\Phi\!: W_{(a_0,a_2)}\lra  V_{(a_0,a_2)}$$
is a homeomorphism.
\end{lmm}

\begin{proof} By~\eref{Flim_e}, 
$$\Phi(x_0,a_0',a_2')=(x_0,a_0',a_2') \qquad\forall~(a_0',a_2')\in W.$$
Thus, $\Phi$ extends to a continuous~map
$$\wt\Phi\!:\wt{W}\!\cup\!\big([x_0,\i)\!\times\!W\big)\lra\R^3, \qquad
\wt\Phi(z,a_0',a_2')=\begin{cases}
\Phi(z,a_0',a_2'),&\hbox{if}~z\!\le\!x_0;\\
(z,a_0',a_2'),&\hbox{if}~z\!\ge\!x_0.
\end{cases}$$
This map is injective on a neighborhood of $(x_0,a_0,a_2)$.
By \cite[Theorem~36.5]{Mu2}, there thus exist neighborhoods $W_{(a_0,a_2)}'$ and $V_{(a_0,a_2)}'$
of $(x_0,a_0,a_2)$ in~$\R^3$ such that~$\wt\Phi$ takes $W_{(a_0,a_2)}'$ 
homeomorphically onto~$V_{(a_0,a_2)}'$.
Taking 
$$W_{(a_0,a_2)}=W_{(a_0,a_2)}' \cap \big((-\i,x_0]\!\times\!\R^2\big)
\quad\hbox{and}\quad
V_{(a_0,a_2)}=V_{(a_0,a_2)}' \cap \big((-\i,x_0]\!\times\!\R^2\big),$$
we complete the proof.
\end{proof}

\begin{crl}\label{homeom_crl} Let $(a_0,a_2)\!\in\!\R^2$ and $G$ be a solution 
of~\eref{F0cond_e} as in~\eref{F0uniqlmm_e0} satisfying~\eref{F0cond_e2} such~that 
\BE{homeomcrl_e} G(x_0)=a_0, \quad 
\lim_{x\lra^-x_0}\!\!\!\!\!G'(x)=a_2, \quad 
\lim_{x\lra^-x_0}\!\!\!\!\!G'''(x)=\i.\EE
Then $(a_0,a_2)\!\in\!W$ and there exists $\de'\!\in\!(0,\de)$ such that
$G\!=\!F_{a_0,a_2}$ on $(x_0\!-\!\de',x_0)$.
\end{crl}

\begin{proof}
By the reasoning in Section~\ref{FI_subs}, the assumptions on~$G$ imply that $(a_0,a_2)\!\in\!W$.
Let~$\Phi$ be as in Lemma~\ref{homeom_lmm}.
By~\eref{homeomcrl_e}, there exist
$x^*\!\in\!\R$ and $(a_0^*,a_2^*)\!\in\!\R^2$ such~that
$$x^*\in(x_0\!-\!\de,x_0),\quad (x^*,a_0^*,a_2^*)\in W_{(a_0,a_2)}, \quad
\big(x^*,G(x^*),G'(x^*)\big)=\Phi(x^*,a_0^*,a_2^*)\,.$$
If $F_{a_0^*,a_2^*}''(x^*)\!<\!G''(x^*)$, there exist
\begin{gather*}
x'\in\!(x^*,x_0), \quad (x',a_0',a_2')\in W_{(a_0,a_2)} \qquad\hbox{s.t.}\\
\Phi(x',a_0',a_2')=\big(x',G(x^*),G'(x^*)\big), \quad
F_{a_0',a_2'}''(x')\!<\!G''(x^*).
\end{gather*}
This contradicts Corollary~\ref{F0uniq_crl}.
If $F_{a_0^*,a_2^*}''(x^*)\!>\!G''(x^*)$,  there exist
\begin{gather*}
x'\in\big(x_0\!-\!\de,x^*\big),  \quad (x',a_0',a_2')\in W_{(a_0,a_2)} \qquad\hbox{s.t.}\\
\Phi(x',a_0',a_2')=\big(x',G(x^*),G'(x^*)\big),
\quad F_{a_0',a_2'}''(x')\!>\!G''(x^*).
\end{gather*}
This contradicts Corollary~\ref{F0uniq_crl} with~$x^*$ and~$x'$ interchanged.\\

\noindent
We conclude~that $F_{a_0^*,a_2^*}''(x^*)\!=\!G''(x^*)$.
Since 
$$9+2F_{a_0^*,a_2^*}'(x^*)-3F_{a_0^*,a_2^*}''(x^*) >0,$$
the uniqueness of solutions of differential equations implies that 
$G\!=\!F_{a_0^*,a_2^*}$ on $[x^*,x_0]$.
By~\eref{homeomcrl_e} and~\eref{Flim_e},  $(a_0^*,a_2^*)\!=\!(a_0,a_2)$.
\end{proof}

\section{Other observations}
\label{misc_sec}

\noindent
We conclude by elaborating two separate points brought up earlier  
in this note.

\subsection{Counts of curves in $\P^3$}
\label{P3g0_subs}

\noindent
We first apply the reasoning of~\cite{FI} described in Section~\ref{FI0_subs} 
to obtain a coarse upper bound for  the numbers of Conjecture~\ref{P3g0_cnj}.\\

\noindent
For  $d\!\in\!\Z^+$ and $p\!\in\!\Z^{\ge0}$, let
$$n_{0,d}(p)=\frac{N_{0,d}(p)}{(2d\!+\!p)!}\,.$$
For $d_1,d_2\!\in\!\Z^+$ and $p_1,p_2\!\in\!\Z^{\ge0}$ with
\BE{P3bnd_e1} p_1\le 2d_1, \qquad p_2\le 2d_2, \qquad
1<p\!\equiv\!p_1\!+\!p_2< 2d\!\equiv\!2(d_1\!+\!d_2),\EE
define
$$f(d_1,d_2,p_1,p_2)
=\frac{(2d_1\!+\!p_1)!(2d_2\!+\!p_2)!}{(2d\!+\!p)!}
d_2\binom{2d\!-\!p\!-\!1}{2d_1\!-\!p_1\!-\!1}
\Bigg(d_1^2\binom{2p\!-\!2}{2p_1}-d_2^2\binom{2p\!-\!2}{2p_2}\Bigg).$$
Since
\begin{equation*}\begin{split}
\frac{(2d_1\!+\!p_1)(2d_2\!+\!p_2)(2d_2\!+\!p_2\!-\!1)}{(2d\!+\!p)(2d\!+\!p\!-\!1)(2d\!+\!p\!-\!2)}
d_1^2d_2\binom{2d\!-\!p\!-\!1}{2d_1\!-\!p_1\!-\!1}
\binom{2p\!-\!2}{2p_1} &\le 8\frac{d_1^3d_2^3}{d^3}\binom{2d\!+\!p\!-\!3}{2d_1\!+\!p_1\!-\!1}\,,\\
\frac{(2d_1\!+\!p_1)(2d_1\!+\!p_1\!-\!1)(2d_1\!+\!p_1\!-\!2)}{(2d\!+\!p)(2d\!+\!p\!-\!1)(2d\!+\!p\!-\!2)}
d_2^3\binom{2d\!-\!p\!-\!1}{2d_1\!-\!p_1\!-\!1}
\binom{2p\!-\!2}{2p_2} &\le 8\frac{d_1^3d_2^3}{d^3}\binom{2d\!+\!p\!-\!3}{2d_1\!+\!p_1\!-\!3}\,,
\end{split}\end{equation*}
we find that   
\BE{P3bnd_e3}\big|f(d_1,d_2,p_1,p_2)\big|\le 8\frac{d_1^3d_2^3}{d^3} \EE
under the assumptions~\eref{P3bnd_e1}.\\

\noindent
By the recursion of \cite[Theorem~10.4]{RT},
\BE{P3bnd_e5a}\begin{split}
n_{0,d}(0)&=\sum_{\begin{subarray}{c}d_1+d_2=d\\ d_1,d_2\ge1\end{subarray}}
\!\!\!\!\!
\frac{(d_2\!-\!d_1)d_1^2d_2(2d_2\!+\!1)}{d(d\!-\!1)(2d\!-\!1)}n_{0,d_1}(0)n_{0,d_2}(1),\\
n_{0,d}(2d)&=\frac12n_{0,d}(2d\!-\!1)+\!\!
\sum_{\begin{subarray}{c}d_1+d_2=d\\ d_1,d_2\ge1\end{subarray}}\!\!\!\!\!
\frac{d_1d_2(4dd_1d_2\!-\!d^2\!+\!2d_1d_2)}{2d(2d\!-\!1)(4d\!-\!1)}
n_{0,d_1}(2d_1)n_{0,d_2}(2d_2)
\end{split}\EE
for $d\!\ge\!2$ and $d\!\ge\!1$, respectively.
For $d\!\ge\!1$ and $1\!\le\!p\!\le\!2d\!-\!1$, this recursion 
with $j_1\!=\!3$ and $j_2\!=\!j_3\!=\!2$ gives
\BE{P3bnd_e5c}
n_{0,d}(p)=\frac{d}{2d\!+\!p} n_{0,d}(p\!-\!1)
+\!\sum_{\begin{subarray}{c}d_1+d_2=d\\ d_1,d_2\ge1\end{subarray}}
\sum_{\begin{subarray}{c}p_1+p_2=p\\ 0\le p_i\le 2d_i\end{subarray}}\!\!\!\!\!\!
f(d_1,d_2,p_1,p_2)\,n_{0,d_1}(p_1)n_{0,d_2}(p_2)\,;\EE
the assumption $j_i\!\ge\!j_{i+1}$ in~\cite{RT} is not essential.
For $p\!=\!1$, \eref{P3bnd_e5c} simplifies to
\BE{P3bnd_e5d}
n_{0,d}(1)=\frac{d}{2d\!+\!1} n_{0,d}(0)+
\sum_{\begin{subarray}{c}d_1+d_2=d\\ d_1,d_2\ge1\end{subarray}}\!
\frac{d_1^3d_2(2d_2\!+\!1)}{d(2d\!-\!1)(2d\!+\!1)}n_{0,d_1}(0)n_{0,d_2}(1)\,.\EE

\vspace{.2in}

\noindent
Let $\wt{n}_1(0)\!=\!1/2$ and define
\begin{equation*}\begin{aligned}
\wt{n}_d(0)&= 4\!\!\!\sum_{\begin{subarray}{c}d_1+d_2=d\\ d_1,d_2\ge1\end{subarray}}
\!\!\!\frac{d_1^3d_2^3}{d^3}\wt{n}_{d_1}(0)\wt{n}_{d_2}(1)
&\quad\forall~&d\!\ge\!2,\\
\wt{n}_d(1)&=\frac{1}{2} \wt{n}_d(0)+
\sum_{\begin{subarray}{c}d_1+d_2=d\\ d_1,d_2\ge1\end{subarray}}\!\!\!
\frac{d_1^3d_2^3}{d^3}\wt{n}_{d_1}(0)\wt{n}_{d_2}(1)
&\quad\forall~&d\!\ge\!1,\\
\wt{n}_d(p)&=\frac{1}{2} \wt{n}_d(p\!-\!1)
+8\!\!\!\!\sum_{\begin{subarray}{c}d_1+d_2=d\\ d_1,d_2\ge1\end{subarray}}
\sum_{\begin{subarray}{c}p_1+p_2=p\\ 0\le p_i\le 2d_i\end{subarray}}\!\!\!\!\!\!
\frac{d_1^3d_2^3}{d^3}\wt{n}_{d_1}(p_1)\wt{n}_{d_2}(p_2)&\quad\forall~&
d\!\ge\!1,~p\!\ge\!2\,.
\end{aligned}\end{equation*}
We also define $\wt{n}_1'(0)\!=\!1/2$ and 
\BE{P3bnd_e9}\begin{aligned}
\wt{n}_d'(0)&= 8\!\!\!\sum_{\begin{subarray}{c}d_1+d_2=d\\ d_1,d_2\ge1\end{subarray}}
\!\!\!\wt{n}_{d_1}'(0)\wt{n}_{d_2}'(0)
&\quad\forall~&d\!\ge\!2,\\
\wt{n}_d'(p)&=\frac{1}{2} \wt{n}_d'(p\!-\!1)
+8\!\!\!\!\sum_{\begin{subarray}{c}d_1+d_2=d\\ d_1,d_2\ge1\end{subarray}}
\sum_{\begin{subarray}{c}p_1+p_2=p\\ 0\le p_i\le 2d_i\end{subarray}}\!\!\!\!\!\!
\wt{n}_{d_1}'(p_1)\wt{n}_{d_2}'(p_2)&\quad\forall~&
d\!\ge\!1,~p\!\ge\!1\,.
\end{aligned}\EE
By \eref{P3bnd_e3}-\eref{P3bnd_e5d}, 
\BE{P3bnd_e7} n_{0,d}(p)\le \wt{n}_d(p)\le \wt{n}_d'(p)\big/d^3 
\qquad\forall~d\!\in\!\Z^+,\,p\!\in\!\Z^{\ge0}.\EE

\vspace{.2in}

\noindent
By~\eref{P3bnd_e9}, the power series 
$$f(x)\equiv \sum_{d=1}^{\i}\ti{n}_d'(0)x^d \qquad\hbox{and}\qquad
g(x,y)\equiv \sum_{d=1}^{\i}\sum_{p=0}^{\i}\ti{n}_d'(p)x^dy^p$$
satisfy
$$f(x)=\frac{1}{2}x+8f(x)^2,\qquad
g(x,y)-f(x)=\frac12yg(x,y)+8\big(g(x,y)^2-f(x)^2\big).$$
Combining these two equations, we obtain
$$8g(x,y)^2-\big(1\!-\!y/2\big)g(x,y)+\frac12x=0.$$
Solving the last equation, we find~that 
\begin{equation*}\begin{split}
16g(x,y)&=(1\!-\!y/2)-\sqrt{(1\!-\!y/2)^2-32x}\Big)
=\big(1\!-\!y/2\big)\Big(1-\sqrt{1-16x/(1\!-\!y/2)^2}\Big)\\
&=2\big(1\!-\!y/2\big)
\sum_{d=1}^{\i}\frac{(2d\!-\!2)!}{d!(d\!-\!1)!}\bigg(\frac{4x}{(1\!-\!y/2)^2}\bigg)^d\\
&=2\big(1\!-\!y/2\big)
\sum_{d=1}^{\i}\sum_{p=0}^{\i}\frac{(2d\!-\!2)!}{d!(d\!-\!1)!}
4^d2^{-p}\binom{2d\!-1\!+\!p}{p}x^dy^p\,.\\
\end{split}\end{equation*}
Since the binomial coefficient above increases with~$p$, 
$$\ti{n}_d'(p)\le \frac{1}{8d}\binom{2d\!-\!2}{d\!-\!1}
4^d \binom{4d\!-\!1}{2d}
\le \frac{1}{8d}2^{2d-2}2^{2d}2^{4d-1}
\le \frac{1}{d} 2^{8d} \qquad\forall\,p\!\le\!2d.$$
Combining with~\eref{P3bnd_e7}, we conclude that 
\BE{g0P3bnd_e} n_{0,d}(p) \le 2^{8d}d^{-4}
\qquad\forall~d,p\!\in\!\Z^+,\,p\!\in\!\Z^{\ge0}\,.\EE
This is an analogue (and a very rough one) of the upper bound in~\eref{P2asymp_e0}.
A lower bound for the sequences of Conjecture~\ref{P3g0_cnj} is more elusive
because the recursions~\eref{P3bnd_e5a} and~\eref{P3bnd_e5c} involve negative coefficients.

\subsection{On recursively defined sequences}
\label{recseq_subs}

\noindent
As indicated above Corollary~\ref{F1_crl} in Section~\ref{FI_subs}, 
the asymptotics for Gromov-Witten invariants in some basic cases 
may reflect the behavior of more general recursively defined 
sequences of the form~\eref{recdfnseq_e}.

\begin{cnj}\label{asympgrowth_cnj}
Suppose $f\!\in\!\Q(d_1,d_2)$ is a rational function defined on $\Z^+\!\times\!\Z^+$
and satisfying
$$f(d_1,d_2)\!>\!0~~\forall\,d_1,d_2\!\in\!\Z^+, \qquad
\lim_{d_1,d_2\lra\i} \Bigg(f(d_1,d_2)\bigg/\bigg(\frac{d_1d_2}{d}\bigg)^k\Bigg)\in\R^+$$
for some $k\!\in\!\R^+$. 
If $n_d$ is a sequence recursively defined by~\eref{recdfnseq_e} and $n_1\!>\!0$,
then the sequence~$\sqrt[d]{n_d}$ is eventually increasing, 
i.e.~there exists $d^*\!\in\!\Z^+$ such that 
\BE{asympgrowth_e}\sqrt[d]{n_d}\le\sqrt[d+1]{n_{d+1}} \qquad\forall d\ge d^*.\EE
\end{cnj}

\vspace{.1in}

\noindent
The conjectural dependence of the asymptotic behavior of~$n_d$ 
only on the asymptotic behavior
of $f(d_1,d_2)$ may be related to the following property.
Let $p(q)\in q\R[q]$ be a polynomial with {\it positive} coefficients and vanishing constant term.
Define the numbers $n_d$ by 
$$\sum_{d=1}^{\i}n_dq^d=\sum_{d=1}^{\i}
\frac{(2d\!-\!2)!}{d!(d\!-\!1)!}
\frac{1}{ad^k}q^d\big(1+p(q)\big)^d\,.$$
It appears that the numbers $\sqrt[d]{n_d}$ are eventually increasing.
In other words, this property is invariant under the change of variables
$$q\lra \big(1+p(q)\big)q$$
if $p(q)$ is a polynomial with {\it positive} coefficients and vanishing constant term.
For the asymptotic behavior conclusion, $p(q)$ would perhaps need to be a power
series with coefficients declining sufficiently quickly.\\

\noindent
We now confirm Conjecture~\ref{asympgrowth_cnj} in the model cases, i.e.~for
$$f(d_1,d_2)=a\frac{d_1^kd_2^k}{d^k}
\qquad\hbox{with}\quad a\!\in\!\R^+,~k\in\R^{\ge0}\,.$$
The crucial problem is how to reduce more general cases of this conjecture to
the model ones.\\

\noindent
By Lemma~\ref{rec_lmm}, 
$$n_d=\frac{(2d\!-\!2)!}{d!(d\!-\!1)!}
\frac{1}{ad^k}\big(n_1a\big)^d\qquad\forall~d\!\in\!\Z^+\,.$$
Thus, the eventually increasing property in this case is equivalent to
the existence of \hbox{$d^*\!\in\!\Z^+$} such~that 
$$\frac{\sum\limits_{r=d+1}^{2d-2}\!\!\!\ln r-\sum\limits_{r=1}^{d-1}\!\ln r-\ln a-k\ln d}{d}
<\frac{\sum\limits_{r=d+2}^{2d}\!\!\!\!\ln r-\sum\limits_{r=1}^{d}\!\ln r-\ln a-k\ln (d\!+\!1)}{d+1}
\qquad\forall~d\ge d^*.$$
This is equivalent to
\BE{lnbnd_e1}d\ln(d\!+\!1)+\sum\limits_{r=d+1}^{2d-2}\!\!\!\ln r-
\sum\limits_{r=2}^{d-1}\ln r-\ln a
+k\big(d\ln(1+\frac1d)  -\ln d\big)<d\ln(2d-1)+d\ln2.\EE
Since $\ln x$ is an increasing function,
\begin{equation*}\begin{split}
\sum_{r=d+1}^{2d-2}\!\!\!\ln r&< 
\int_{d+1}^{2d-1}\!\!\!\ln x\,\d x =(x\ln x-x)\Big|_{d+1}^{2d-1}
=(2d\!-\!1)\ln(2d\!-\!1)-(d\!+\!1)\ln(d\!+\!1)-(d\!-\!2),\\
\sum_{r=2}^{d-1}\ln r&>
\int_{1}^{d-1}\!\!\!\ln x\,\d x=(x\ln x-x)\Big|_1^{d-1}
=(d\!-\!1)\ln(d\!-\!1)-(d\!-\!2).
\end{split}\end{equation*}
Thus, the left-hand side of \eref{lnbnd_e1} is bounded by
\begin{equation*}\begin{split}
&(2d\!-\!1)\ln(2d\!-\!1)-(d\!-\!1)\ln(d\!-\!1)-\ln(d\!+\!1)-\ln a
-k\big(\ln d-1\big)\\
&\qquad \le d\ln(2d\!-\!1)+(d\!-\!1)\ln2+(d\!-\!1)\ln\bigg(1+\frac1{2(d\!-\!1)}\bigg)
-\ln(d\!+\!1)-\ln a\\
&\qquad \le d\ln(2d\!-\!1)+(d\!-\!1)\ln2+\frac12-\ln(d\!+\!1)-\ln a
\le d\ln(2d\!-\!1)+d\ln2-\ln(d\!+\!1)-\ln a.
\end{split}\end{equation*}
For $d$ sufficiently large, the combination of the last two terms above is negative,
which establishes~\eref{asympgrowth_e} in this case.\\

\vspace{.2in}

\noindent
{\it Department of Mathematics, Stony Brook University, Stony Brook, NY 11794\\
azinger@math.stonybrook.edu}

\end{document}